\documentclass[a4paper, 12pt, oneside]{amsart}
\usepackage{amsmath,amssymb,graphicx,setspace,verbatim,url}
\usepackage{mathtools}
\usepackage{adjustbox}
\usepackage{color}
\usepackage{hyperref}
\usepackage{multicol}
\usepackage{enumitem}
\usepackage{floatrow}
\usepackage{tikz}
\usepackage{tikz-cd}
\usepackage{multicol}
\usepackage{booktabs}

\usetikzlibrary{positioning}
\usepackage[margin=2.5cm]{geometry}
\usepackage[english]{babel}

\input xy
\xyoption{all}
\newtheorem{prop}{Proposition}[section]
\newtheorem{coro}[prop]{Corollary}
\newtheorem{thm}[prop]{Theorem}
\newtheorem*{thm*}{Theorem}

\newtheorem{lemma}[prop]{Lemma}

\newtheorem*{conjecture*}{Conjecture}

\theoremstyle{definition}
\newtheorem*{ackn}{Acknowledgments}

\newcommand{\D}{\mathcal{D}}

\newcommand{\GL}{\mathrm{GL}}

\newcommand{\pt}{\mathrm{pt}}
\DeclareMathOperator{\AGL}{AGL}
\DeclareMathOperator{\AU}{AU}
\DeclareMathOperator{\ASp}{ASp}
\DeclareMathOperator{\AO}{AO}
\DeclareMathOperator{\AX}{AX}
\DeclareMathOperator{\U}{U}
\DeclareMathOperator{\Sp}{Sp}
\DeclareMathOperator{\Or}{O}
\DeclareMathOperator{\X}{X}
\DeclareMathOperator{\rank}{rank}
\DeclareMathOperator{\sign}{sign}
\usepackage{thmtools}
\usepackage{xcolor}
\definecolor{myblue}{HTML}{007ACC}
\hypersetup{
    colorlinks,
    linkcolor={myblue},
    citecolor={myblue},
    urlcolor={myblue}
    }

\title{On the proportion of derangements \\in affine classical groups}
\author[J.~Anzanello]{Jessica Anzanello}
\address{Dipartimento di Matematica e Applicazioni, University of Milano-Bicocca, Via Cozzi 55, 20125 Milano, Italy
}
\email{j.anzanello@campus.unimib.it}
\date{}
\begin{document}


\begin{abstract}
We derive exact formulas for the proportions of derangements and of derangements of $p$-power order in the affine classical groups $\AU_m(q)$, $\ASp_{2m}(q)$, $\AO_{2m+1}(q)$ and $\AO^{\pm}_{2m}(q)$, where $p$ denotes the characteristic of the defining finite field.

In the unitary case, the formulas rely on a result on partitions of independent interest: we obtain a generating function for integer partitions $\lambda=(\lambda_1, \dots, \lambda_m)$ into $m$ parts, with $\lambda_1\ge \dots \ge \lambda_m$, such that either $\lambda_1=1$ or $\lambda_{k-1}>\lambda_k=k$ for some $k \in \{2, \dots,m\}$.

In the symplectic and orthogonal cases, the proofs of the formulas reduce to verifying three $q$-polynomial identities conjectured by the author and later proved by Fulman and Stanton.
\end{abstract}
\date{First version: August 2025; revised version: February 2026}
\maketitle

\tableofcontents
\section{Introduction}
Let $G$ be a non-trivial finite transitive permutation group on a set $\Omega$. An element $g \in G$ is a \textit{derangement} if it has no fixed points on $\Omega$, that is, if $\alpha^g \neq \alpha$ for any $\alpha \in \Omega$. 

Let $\mathcal{D}(G)$ be the set of derangements of $G$ and let $$\delta(G) \coloneqq \frac{|\D(G)|}{|G|}$$ 
denote the proportion of derangements in $G$, which one can view as the probability that a randomly chosen element of $G$ is a derangement.
A result going back to Jordan asserts that $\delta(G)>0$, and, as discussed by Serre in \cite{Serre}, Jordan's theorem has significant implications in diverse areas, ranging from number theory to topology.

The study of derangements in transitive permutation groups has a long history: in 1708, Montmort, in his foundational work on probability theory \cite{Montmort},  introduced the well-known ``inclusion-exclusion principle'' to obtain the following formula for the proportion of derangements in the symmetric group $S_n$ of degree $n$:
\begin{equation*}
    \delta(S_n)=\sum_{i=0}^{n} \frac{(-1)^i}{i!}.
\end{equation*}

In more recent years, an extensive body of research has developed around counting derangements and investigating derangements with specific properties, such as prescribed order. For a comprehensive survey on the subject, we refer the reader to the introductory chapter of \cite{BurGiu}.

In \cite{SpigaAGL}, Spiga obtained a notably simple and explicit formula for the proportion of derangements in the affine general linear group $\AGL_m(q)$, which can be viewed as a natural $q$-analogue of the corresponding problem for the symmetric group, namely:
$$
\delta(\AGL_m(q))=\sum_{i=1}^{m}\frac{(-1)^{i-1}}{q^{i(i+1)/2}}.
$$
Let $\U_m(q) \le \GL_m(q^2)$, $\Sp_{2m}(q) \le \GL_{2m}(q)$, 
$\Or_{2m+1}(q) \le \GL_{2m+1}(q)$ and $\Or^{\pm}_{2m}(q) \le \GL_{2m}(q)$ denote the isometry groups of, respectively, a non-degenerate unitary, symplectic, quadratic and quadratic of plus (resp. minus) type form over the field $\mathbb{F}=\mathbb{F}_{q^e}$, where $e=2$ in the unitary case, and $e=1$ otherwise. Moreover, let $p$ denote the characteristic of $\mathbb{F}$. In this article we establish exact formulas for the proportions of derangements, as well as for derangements of $p$-power order, in the affine classical groups 
\begin{equation}
\label{affineclassical}
\AU_m(q)\text{, } \ASp_{2m}(q)\text{, } \AO_{2m+1}(q) \text{ and }\AO^{\pm}_{2m}(q),
\end{equation} 
in their natural actions.
The main result obtained is the following theorem.
\begin{thm}
\label{main}
The following formulas hold.
\begin{enumerate}
\item
\label{Unitary} 
$\displaystyle \delta(\AU_m(q))=\frac{1}{q+1}\left(1-\frac{1}{(-q)^{m(m+3)/2}}\right);$
\item
\label{Symplectic} 
$\displaystyle \delta(\ASp_{2m}(q))=\frac{1}{q+1}\left(1-\frac{1}{(-q)^{m(m+2)}}\right);$
\item
\label{Orthogonal_odd}
$\displaystyle \delta(\AO_{2m+1}(q))=\frac{1}{2}+ \frac{(-1)^{m-1}}{2q^{(m+1)^2}};$
\item
\label{Orthogonal_even}
$\displaystyle \delta(\AO^{\pm}_{2m}(q))=\frac{1}{2} \pm \frac{(-1)^{m-1}}{2q^{m(m+1)}}$.
\end{enumerate}
\end{thm}

Given a non-negative integer $j$ and an indeterminate $x$, we define
\begin{equation}
\label{(x)_j}
(x)_j \coloneqq \begin{cases}
    1 & \text{ if }j=0, \\
    \prod_{i=1}^j (1-x^i)=(1-x)(1-x^2)\cdots(1-x^j) & \text{ if } j \ge 1.
\end{cases}
\end{equation}
Further, if $G$ is one of the affine classical groups listed in  Equation \eqref{affineclassical}, we denote by $\delta_p(G)$ the proportion of derangements of $p$-power order in $G$: this quantity plays a crucial role in the computation of the exact formulas for $\delta(G)$, and we obtain the following result.

\begin{thm}
\label{p_main}
The following formulas hold.
\begin{enumerate}
\item \label{p_Unitary}
$\displaystyle \delta_p(\AU_m(q))=
\frac{1}{q^m(q+1)}\sum_{i=1}^m\frac{(-1)^i((-q)^{i+1}-1)}{(-q)^{i(i+1)/2}(-1/q)_{m-i}};$
\item \label{p_Symplectic}
$\displaystyle \delta_p(\ASp_{2m}(q))=\frac{1}{q^m(q+1)}\sum_{i=1}^{m}\frac{(-1)^{i-1}(q^{2i+1}+1)}{q^{i(i+1)}(1/q^2)_{m-i}};$
\item \label{p_Orthogonal_2m+1}
$\displaystyle \delta_p(\AO_{2m+1}(q))=\frac{1}{2q^m(1/q^2)_{m}}+\frac{1}{2q^{m+1}}\sum_{i=0}^{m}\frac{(-1)^{i-1}}{q^{i(i+1)}(1/q^2)_{m-i}}$, with $p$ odd;
\item \label{p_Orthogonal_2modd}
$\displaystyle \delta_p(\AO^{\pm}_{2m}(q))=\frac{q^m\pm1}{2q^{2m}}\sum_{i=1}^{m}\frac{(-1)^{i-1}}{q^{i(i-1)}(1/q^2)_{m-i}}$, with $p$ odd;
\item \label{p_Orthogonal_2meven}
$\displaystyle \delta_2(\AO^{\pm}_{2m}(q))=\frac{1}{2q^{m-1}(1/q^2)_{m-1}}\pm \frac{1}{2q^{2m}}\sum_{i=1}^{m}\frac{(-1)^{i-1}}{q^{i(i-1)}(1/q^2)_{m-i}}.$
\end{enumerate}
\end{thm}

The proof of Theorem \ref{p_main}\eqref{p_Unitary} relies on delicate arguments involving integer partitions. 
In particular, we derive a generating function for the following class of partitions.

Let $$\lambda=(\lambda_1,\dots,\lambda_m)$$ denote a partition of some non-negative integer $|\lambda|=\sum_{i=1}^{m}\lambda_i$ into $m$ parts $\lambda_1 \ge \lambda_2 \ge \dots \ge \lambda_m>0$.
For $m \in \mathbb{N}$, let
\begin{equation}
\label{Lambda_m}
\Lambda_m\coloneqq\{\lambda=(\lambda_1,\dots,\lambda_m)\mid \lambda_1=1 \text{ or } \lambda_{k-1}>\lambda_k=k\text{ for some }k \in \{2,\dots,m\} \}.
\end{equation}

\begin{thm}
\label{gen_fun_U}
Let $m \in \mathbb{N}$.
The generating function for the number of partitions in $\Lambda_m$ is 
$$\sum_{\lambda \in \Lambda_m}x^{|\lambda|}
=\frac{x^m}{1-x}\sum_{i=1}^{m}\frac{(-1)^{i}x^{i(i-1)/2}(x^i-1)}{(x)_{m-i}}.
$$
\end{thm}

Table ~\ref{exgenfun} provides a few small examples illustrating the content of Theorem \ref{gen_fun_U}.
\begin{table}[h]\centering
\caption{Small examples for Theorem \ref{gen_fun_U} }
\label{exgenfun}
\begin{tabular}{lccc} \toprule
      $m$             & $\Lambda_m$    & $\sum_{\Lambda_m}x^{|\lambda|}$ & first terms in   $\sum_{\Lambda_m}x^{|\lambda|}$  \\ \midrule
$1$ & $(1)$  & $x$   \\[1em]  
$2$  & $\begin{array}{c}(1,1), \\ (k,2) ,\text{ } k\ge 3 \end{array}$  & $x^2+\frac{x^5}{1-x}$ & $x^2+x^5+x^6+x^7+x^8+\dots
$  \\[1.5em]  
$3$  &$\begin{array}{c} (1,1,1),\\ 
(k,2,1),\text{ }k\ge3, \\
(k,2,2),\text{ }k\ge3, \\(j,k,3),\text{ }j\ge k\ge4
\end{array}$  &$\frac{x^3}{1-x}\left(\frac{1}{1-x^2}-x-x^2+x^3-x^6\right)$ & $\begin{array}{l} x^3+x^6+2x^7+2x^8+2x^9+\\2x^{10}+
3x^{11}+3x^{12}+4x^{13}+\\4x^{14}+5x^{15}+\dots
\end{array}$ \\\bottomrule
\end{tabular}
\end{table}

Regarding the affine symplectic and orthogonal groups, the proofs of Theorem \ref{p_main}~\eqref{p_Symplectic}-\eqref{p_Orthogonal_2meven} reduce to the verification of three $q$-polynomial identities, stated in Theorem \ref{conj_identities}. These identities were conjectured by the author in the first version of this paper, and posted on the arXiv.
While the present paper was under revision, they were proved by Fulman and Stanton in \cite{FulmanStanton25}, and therefore now appear as theorems in the final version.
Their proof reveals interesting connections with symplectic and orthogonal Cohen-Lenstra type distributions and with hypergeometric series. Before stating the identities, we fix some notation for integer partitions. Let $\lambda$ be a partition of a non-negative integer $|\lambda|$. We denote by $m_i(\lambda)$ the number of parts of size $i$, and by $\lambda'_i=\sum_{j\ge i}m_j(\lambda)$ the number of parts of size at least $i$. Furthermore, we let $o(\lambda)$ denote the number of odd parts of $\lambda$. 

\begin{thm}[Fulman and Stanton, \cite{FulmanStanton25}]
\label{conj_identities}
The following identities hold.
\begin{enumerate}
\item
\label{conj_intro_sympl}
$\displaystyle\sum_{\substack{|\lambda|=2m\\ i \text{ odd }\Rightarrow m_i(\lambda) \text{ even}}}\frac{1-q^{-\lambda'_1}}{q^{\frac{1}{2}\sum_i(\lambda'_i)^2+\frac{1}{2}o(\lambda)}\prod_i(1/q^2)_{\lfloor\frac{m_i(\lambda)}{2}\rfloor}}=
\frac{1}{q^m(q+1)}\sum_{i=1}^{m}\frac{(-1)^{i-1}(q^{2i+1}+1)}{q^{i(i+1)}(1/q^2)_{m-i}};
$
\item 
\label{conj_intro_ort_2m+1}
$\displaystyle \sum_{\substack{|\lambda|=2m+1\\i \text{ even}\Rightarrow \\m_i(\lambda) \text{ even}}}\frac{1-q^{-\lambda'_1}}{q^{\frac{1}{2}\sum_i (\lambda'_i)^2-\frac{1}{2}o(\lambda)}\prod_i(1/q^2)_{\lfloor\frac{m_i(\lambda)}{2}\rfloor}}=\frac{1}{q^m(1/q^2)_{m}}+\frac{1}{q^{m+1}}\sum_{i=0}^{m}\frac{(-1)^{i-1}}{q^{i(i+1)}(1/q^2)_{m-i}};$

\item 
\label{conj_intro_ort_2m}
$\displaystyle \sum_{\substack{|\lambda|=2m\\i \text{ even}\Rightarrow m_i(\lambda) \text{ even}}}\frac{1-q^{-\lambda'_1}}{q^{\frac{1}{2}\sum_i (\lambda'_i)^2-\frac{1}{2}o(\lambda)}\prod_i(1/q^2)_{\lfloor\frac{m_i(\lambda)}{2}\rfloor}}
=\frac{1}{q^m}\sum_{i=1}^{m}\frac{(-1)^{i-1}}{q^{i(i-1)}(1/q^2)_{m-i}}.$
\end{enumerate}
\end{thm}

We note that estimates for the proportions of derangements under consideration can be obtained by combining \cite[Proposition~1.1]{GurTiep03} with results in \cite{NP98} on the number of derangements for  classical groups, in their natural action on the set of non-zero vectors of the natural module. 

One of the main tools used to prove Theorems \ref{main} and \ref{p_main}, 
is the cycle index for the finite classical groups, introduced by Fulman in \cite{Fulman_cycle_index}. 
Cycle indices are generating functions that encode useful information for studying random matrices depending only on their conjugacy classes. They constitute a powerful tool for counting derangements and, notably, were a fundamental tool in Fulman and Guralnick's remarkable proof of the Boston-Shalev conjecture, which states the proportion of derangements in a transitive action of a simple group on a set $\Omega$, with $|\Omega|>1$, is uniformly bounded away from 0 (see \cite{FGsimple}, \cite{FGChev}, \cite{FGsub} and \cite{FGext}). We will recall the definitions of cycle index for each classical group at the beginning of the corresponding section.

This work builds on the approach developed for $\AGL_m(q)$ in \cite{SpigaAGL}, and we adopt the same notation to ensure consistency.

\begin{ackn}
I thank my PhD supervisor, Pablo Spiga, for his thoughtful guidance and constant support. I am also grateful to Fedor Petrov, for providing the proof of Lemma \ref{|A|=|B|}, to Tewodros Amdeberhan, for useful comments which contributed to the proof of Proposition \ref{prop_bar{u}}, and to Jason Fulman, for helpful discussions on the even characteristic case. I also thank the anonymous referee for valuable comments and suggestions, which have improved the clarity of the paper.

The author is a member of GNSAGA.
\end{ackn}
\section{Preliminaries}
\label{preliminaries}
\subsection*{Group theory preliminaries} 

Let $m$ be a positive integer and let $q=p^f$ be a prime power. Let $$\X_m(q) \in \{\U_m(q), \Sp_{m}(q), \Or^{\epsilon}_m(q)\}, \text{ }\epsilon \in \{+,-,\circ\}.$$ Moreover, let $V$ be the natural module for $\X_m(q)$, so that $V=\mathbb{F}_{q^e}^{m}$, where $e=2$ if $\X_m(q)=\U_m(q)$ and $e=1$ otherwise.
We consider the group $\AX_m(q)$, the semidirect product of the normal subgroup $V$ and $\X_m(q)$, and we refer to it as an \textit{affine unitary} (resp. \textit{symplectic, orthogonal}) \textit{group}. We collectively refer to these groups as \textit{affine classical groups}.
We denote elements of $\AX_m(q)$ as follows: for each matrix $a \in \X_m(q)$ and vector $v \in V$, we define the permutation $\varphi_{a,v}: V \rightarrow V$ by
$$\varphi_{a,v}: u \mapsto ua+v.$$

In the introduction, we defined $\delta_p(\AX_m(q))$ as the proportion of derangements of $\AX_m(q)$ of $p$-power order. Now, note that $\varphi_{a,v} \in \AX_m(q)$ is a $p$-element if and only if $a$ is a unipotent matrix. Hence, if we define
   $$ \mathcal{U}(\AX_m(q)) := \{\varphi_{a,v} \in \mathcal{D}(\AX_m(q)) \mid a \text{ unipotent}\},$$
then
\begin{equation}
\label{unip}
    \delta_p(\AX_m(q))= \frac{|\mathcal{U}(\AX_m(q))|}{|\AX_m(q)|}.
\end{equation}

We now state the following simple observation.
\begin{lemma}
\label{equiv}
Let $\varphi_{a,v} \in \AX_m(q)$. Then $\varphi_{a,v}$ is a derangement if and only if $v \notin \{u(a-1) \mid u \in V\}$. Moreover, $\varphi_{a,v}$ is a derangement of prime power order if and only if $a$ is a unipotent matrix and $v \notin \{u(a-1) \mid u \in V\}$.
\end{lemma}
\begin{proof}
Let $u \in V$. Then $\varphi_{a,v}$ fixes $u$ if and only if
$$u=u^{\varphi_{a,v}}=ua+v,$$
that is, $v \in \{u(a-1) \mid u \in V\}$. The two statements follow immediately.
\end{proof}

In view of Lemma \ref{equiv}, we have:
\begin{align}
    \delta(\AX_m(q))&= \frac{1}{|\AX_m(q)|}\sum_{a \in \X_m(q)}|\mathbb{F}^m_{q^e} \setminus\{u(a-1) \mid u \in \mathbb{F}^m_{q^e}\}|\nonumber  \\
    &= \frac{1}{|\AX_m(q)|}\sum_{a \in \X_m(q)}\left(q^{em}-q^{e\rank(a-1)}\right) \nonumber \\
    &= \frac{q^{em}}{|\AX_m(q)|}\sum_{a \in \X_m(q)}\left(1-q^{-e\dim \ker(a-1)}\right) \nonumber \\
    \label{prop_der}
    &=1 - \frac{1}{|\X_m(q)|}\sum_{a\in \X_m(q)}q^{-e\dim \ker(a-1)}.
\end{align}

Let $\Delta_{u}(\X_m(q))$ denote the proportion of unipotent elements of $\X_m(q)$. Applying again Lemma \ref{equiv}, an entirely similar computation yields:
\begin{align}
    \delta_p(\AX_m(q))&=\frac{1}{|\X_m(q)|}\sum_{\substack{a \in \X_m(q),\\ a \text{ unipotent}}}\left(1-q^{-e \dim \ker(a-1)}\right) \nonumber \\
    \label{prop_unip_der}
    &=\Delta_{u}(\X_m(q))-\frac{1}{|\X_m(q)|}\sum_{\substack{a\in \X_m(q),\\a \text{ unipotent}}}q^{-e\dim \ker(a-1)}.
\end{align}
In view of Equations \eqref{prop_der} and \eqref{prop_unip_der}, we will also be interested in the following quantities:
\begin{align} 
\delta'(\X_m(q))&=1-\delta(\AX_m(q))=\frac{1}{|\X_m(q)|}\sum_{a\in \X_m(q)}q^{-e\dim \ker(a-1)}, \label{d'(X_m(q))}\\
\label{d'_p(X_m(q))}
\delta'_p(\X_m(q))&=\Delta_u(\X_m(q))-\delta_p(\AX_m(q))=\frac{1}{|\X_m(q)|}\sum_{\substack{a\in \X_m(q),\\a \text{ unipotent}}}q^{-e\dim \ker(a-1)}.
\end{align}

Equations \eqref{prop_unip_der} and \eqref{d'_p(X_m(q))} involve the proportion of unipotent elements in $\X_m(q)$.
In \cite[Theorem~15.2]{Steinberg}, Steinberg proved that 
if $G$ is a connected reductive group and $\sigma$ is a bijective endomorphism of $G$ such that the pointwise stabiliser $G_{\sigma}$ is finite, then the number of unipotent elements (equivalently, of $p$-elements) of $G_{\sigma}$ equals the square of the order of a Sylow $p$-subgroup, where $p$ denotes the characteristic of the defining finite field. This determines the number of unipotent element in $\U_{m}(q)$ and $\Sp_{2m}(q)$. Instead, orthogonal groups are disconnected, so Steinberg's result is not directly applicable. Nevertheless, in odd characteristic, unipotent elements of $\Or^{\epsilon}_m(q)$ always lie in $\Omega^{\epsilon}_m(q)$, so Steinberg’s theorem does imply that the number of unipotent elements in $\Or^{\epsilon}_m(q)$, with $q$ odd, is the square of the order of a $p$-Sylow. Finally, the number of 2-elements in $\Or^{\pm}_{2m}(q)$ with $q$ even was computed by Fulman and Guralnick in \cite[Proposition~6.11]{FGChev} using the cycle index for orthogonal groups, and equals $q^{2m^2-2m+1}\left(1+\frac{1}{q}\mp \frac{1}{q^m}\right)$.
We collect the values of $\Delta_u(\X_m(q))$ in the following lemma.
\begin{lemma}
\label{steinberg}
The following formulas hold.
\begin{enumerate}
    \item \label{p_u}
    $\displaystyle \Delta_u(\U_m(q))=\frac{1}{q^m(-1/q)_m}$;
    \item \label{p_sp} $\displaystyle \Delta_u(\Sp_{2m}(q))=\frac{1}{q^m(1/q^2)_m}$;
    \item \label{p_o_2m1} $\displaystyle \Delta_u(\Or_{2m+1}(q))=\frac{1}{2q^m(1/q^2)_m}$, if $p$ is odd;
    \item \label{p_o2m} $\displaystyle \Delta_u(\Or^{\pm}_{2m}(q))=
    \begin{cases}
    \displaystyle\frac{q^{m^2-m}}{2(q^m \mp 1)\prod_{i=1}^{m-1}(q^{2i}-1)}, & \text{ if } p \text{ is odd};\\
    \displaystyle\frac{q^{(m-1)^2}(q^m+q^{m-1}\mp1)}{2(q^m \mp 1)\prod_{i=1}^{m-1}(q^{2i}-1)}, & \text{ if } p=2.
    \end{cases}$
\end{enumerate}
\end{lemma}

\subsection*{Partition theory preliminaries}
We conclude this preliminary section by collecting the notation on partitions that will be used throughout the paper. Let $\lambda=(\lambda_1,\dots,\lambda_m)$ be a partition of some non-negative integer $|\lambda|=\sum_{i=1}^{m}\lambda_i$ into $m$ parts $\lambda_1 \ge \dots \ge \lambda_m>0$. We write $\pt(\lambda)$ for the number of parts of $\lambda$, and $m_i(\lambda)$ for the number of parts of size $i$. Let $o(\lambda)$ denote the number of odd parts of $\lambda$, counted with their multiplicity. For each $i \ge 1$, let $\lambda'_i=\sum_{j\ge i}m_j(\lambda)$ be the number of parts of size at least $i$; the sequence $\lambda'=(\lambda'_i)_i$ defines a partition of $|\lambda|$, called the \textit{dual} to $\lambda$. Note, in particular, that $\lambda'_1=\pt(\lambda)$.
It is often convenient to represent partitions diagrammatically by means of the \textit{Ferrers diagram}: to each partition $\lambda=(\lambda_1,\dots,\lambda_m)$ we associate a graphical pattern of dots, in which the $i$-th row contains $\lambda_i$ dots; for example, the Ferrers diagram of the partition $(6,5,4,2,2)$ is represented in Figure \ref{fig:ex_ferrers}.
\begin{figure}
    \centering
    \includegraphics[width=4.5cm]{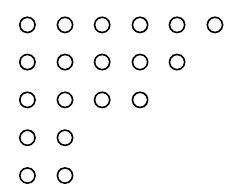}
    \caption{
Ferrers diagram of the partition $(6,5,4,2,2)$
}\label{fig:ex_ferrers}
\end{figure}
A partition $\lambda$ is said to have \textit{Durfee square} $s$ if $s$ is the largest number such that $\lambda'_s \ge s$, that is, $\lambda$ contains at least $s$ parts of size at least $s$. Graphically, the Durfee square corresponds to the side length of the largest square that fits inside the Ferrers diagram of $\lambda$.
For instance, the partition in Figure \ref{fig:ex_ferrers} has Durfee square 3.

Finally, since we will frequently work with generating functions, we remark the following. Let $\mathcal{S}$ be a set of partitions. Unless otherwise specified, by the \textit{generating function for (the number of partitions in)} $\mathcal{S}$ we mean the power series whose coefficient of $x^n$ counts the number of partitions in $\mathcal{S}$ of size $n$; that is 
$$\sum_{\lambda \in\mathcal{S}}x^{|\lambda|}.$$

\section{Affine unitary groups}
\label{sec_unitary}
The unitary group $\U_m(q)$ is the subgroup of $\GL_m(q^2)$ that preserves a non-degenerate unitary form on $\mathbb{F}_{q^2}$.
We begin this section by recalling the definition of cycle index for the finite unitary groups, following the exposition in \cite[Section~4.1]{Fulman_cycle_index}. To this end, we first review the structure of the conjugacy classes of $\U_m(q)$.

Given a polynomial $\phi \in \mathbb{F}_{q^2}[z]$, with non-zero constant term, define the polynomial $\widetilde{\phi}$ by
$$
\widetilde{\phi}=(\phi(0))^{-q}z^{\deg(\phi)}\phi^{q}(z^{-1}),
$$
where $\phi \mapsto \phi^q$ is the map which raises each coefficient of $\phi$ to the $q$th power.

The conjugacy class of an element $a \in \U_m(q)$ is uniquely determined by its rational canonical form and, in \cite{Wall}, Wall proved that it is encoded by the following combinatorial data: to each monic, non-constant, irreducible polynomial $\phi$ over $\mathbb{F}_{q^2}$, the element $a$ associates a partition $\lambda_{\phi}=\lambda_{\phi}(a)$ of a non-negative integer $|\lambda_{\phi}|$, determined by its rational canonical form. 
The collection $(\lambda_{\phi})_{\phi}$ represents a conjugacy class in $\U_m(q)$ if and only if the following conditions are satisfied:

\begin{enumerate}[label=(\roman*)]
    \item $|\lambda_z| = 0$,
    \item $\lambda_{\widetilde{\phi}}=\lambda_{\phi}$,
    \item {$\sum_{\phi} |\lambda_{\phi}| \deg(\phi)=m$}.
\end{enumerate}
From now on, all polynomials $\phi \in \mathbb{F}_{q^2}[z]$ will be assumed to be monic and irreducible.
The \textit{cycle index} $Z_{\U}$ for the unitary groups is the formal power series
$$Z_{\U} := 1 + \sum_{m=1}^{\infty} \frac{y^m}{|\U_m(q)|}\sum_{a \in \U_m(q)} \prod_{\phi \neq z}x_{\phi, \lambda_{\phi}(a)}. $$

Recall the definition of the function $(x)_j$ in Eq. \eqref{(x)_j} and define
\begin{equation*}
c_{\GL,q}(\lambda)\coloneqq q^{ \sum_{i}(\lambda'_i)^2}\prod_i(1/q)_{m_i(\lambda)}.
\end{equation*}
This notation is used to relate the cycle index for the unitary groups to the cycle index of the general linear groups (see \cite{Fulman_cycle_index}).

From \cite[Theorem 10]{Fulman_cycle_index}, we have the following useful factorisation for the cycle index of the unitary groups:
\begin{equation}
\label{fact_Z_u}
Z_{\U}= \prod_{\phi \neq z, \phi=\widetilde{\phi}} \left( \sum_{\lambda}x_{\phi,\lambda}\frac{(-y)^{|\lambda|\deg(\phi)}}{c_{\GL,-q^{\deg(\phi)}}(\lambda)}\right)\cdot \prod_{\substack{\{\phi, \widetilde{\phi}\} \\ \phi \neq \widetilde{\phi}}} \left( \sum_{\lambda}x_{\phi,\lambda}x_{\widetilde{\phi},\lambda}\frac{y^{2|\lambda|\deg(\phi)}}{c_{\GL,q^{2\deg(\phi)}}(\lambda)}\right). 
\end{equation}
In this section, for $m \ge 1$, let $d_m(q)$ and $u_m(q)$ denote respectively $\delta(\AU_m(q))$ and $\delta_p(\AU_m(q))$, and let $d'_m(q)$ and $u'_m(q)$ denote respectively $\delta'(\U_m(q))$ and $\delta'_p(\U_m(q))$. Moreover, we set $d'_0(q)=u'_0(q)\coloneqq 1$.

From Equations \eqref{d'(X_m(q))} and \eqref{d'_p(X_m(q))}, and Lemma \ref{steinberg}\eqref{p_u}, we have
\begin{align*}
d'_m(q)&=\frac{1}{|\U_m(q)|}\sum_{a \in \U_m(q)}q^{-2\dim \ker(a-1)}, \\
d_m(q)&=1-d'_m(q), \\
 u'_m(q)&=\frac{1}{|\U_m(q)|}\sum_{\substack{a \in \U_m(q), \\a \text{ unipotent}}}q^{-2\dim \ker(a-1)}, \\
 u_m(q)&=\frac{1}{q^m(-1/q)_m}-u'_m(q).
 \end{align*}
Finally, we define the generating functions
\begin{align}
D_{\U} \coloneqq \sum_{m=1}^{\infty}d_m(q)y^m, &&
U_{\U} \coloneqq \sum_{m=1}^{\infty}u_m(q)y^m, \\
\label{D'_U}
D'_{\U}\coloneqq\sum_{m=0}^{\infty}d'_m(q)y^m, && U'_{\U} \coloneqq\sum_{m=0}^{\infty}u'_m(q)y^m.
\end{align}

\subsection{Derangements of $p$-power order: proof of Theorem \ref{p_main}\eqref{p_Unitary}}
In this section, we prove Theorem \ref{p_main}\eqref{p_Unitary}, assuming Theorem \ref{gen_fun_U}, which will be established in Section \ref{Sec_gen_fun_U}.
Let $a \in \U_m(q)$ be unipotent, then the conjugacy class of $a$  corresponds to a partition $\lambda_{z-1}(a)=\lambda$ of size $m$.
 \begin{lemma}
 \label{lemma_unip_U}
     $$u_m(q)=(-1)^m\sum_{\substack{\lambda \\ |\lambda|=m}}\frac{1-q^{-2\lambda'_1}}{(-q)^{\sum_{i}(\lambda'_i)^2}\prod_i(-1/q)_{m_i(\lambda)}}.$$
\begin{proof}
The first equality in Eq. \eqref{prop_unip_der} gives
$$u_m(q)=\frac{1}{|\U_m(q)|}\sum_{\substack{a \in \U_m(q)\\a \text{ unipotent}}}\left(1-q^{-2\dim \ker(a-1)}\right).$$
Note that, if $a$ is unipotent and if the partition associated to $z-1$ is $\lambda$, then $\lambda_1^{'}=\dim \ker(a-1)$. Now, it follows from the definition of cycle index for the unitary groups  that $u_m(q)$ is the coefficient of $y^m$ in 
  $$\sum_{\lambda}x_{z-1,\lambda}\frac{(-y)^{|\lambda|}}{(-q)^{\sum_{i}(\lambda'_i)^2}\prod_i(-1/q)_{m_i(\lambda)}},$$ when we substitute all variables $x_{z-1,\lambda}$ with $1-q^{-2\lambda'_1}$.
\end{proof}
 \end{lemma}
Let us define 
\begin{equation}
\label{H}
    H(x)\coloneqq\sum_{|\lambda|=m}\frac{(1-x^{2\lambda'_1})x^{\sum_i(\lambda'_i)^2}}{\prod_i(x)_{m_i(\lambda)}},
\end{equation}
so that 
\begin{equation}
\label{u_m(q)=(-1)^mH(-1/q)}
u_m(q)=(-1)^mH(-1/q).
 \end{equation}
Our strategy to proving Theorem \ref{p_main}\eqref{p_Unitary} is to suitably rewrite $H(x)$. To this end, we first require the following result.
 \begin{thm}[{\cite[Lemma 4.2 and Theorem 1.3]{SpigaAGL}}]
 \label{thmG}
Let $m \in \mathbb{N}$, and let $\Delta_m$ be the set of partitions $\lambda=(\lambda_1,\dots,\lambda_m)$ into $m$  parts such that $\lambda_k=k$ for some $k \in \{1,\dots,m\}$. Let $$G(x)\coloneqq\sum_{\lambda \in \Delta_m}x^{|\lambda|}$$ denote the generating function for the number of partitions in $\Delta_m$.

Then $G(x)$ admits the following equivalent expressions:
$$G(x)=\sum_{|\lambda|=m}\frac{(1-x^{\lambda'_1})x^{\sum_i(\lambda'_i)^2}}{\prod_i(x)_{m_i(\lambda)}}=\sum_{j=0}^{m-1}\frac{x^{(m-j)^2+j}(x)_{m-1}}{(x)^2_{m-j-1}(x)_j}=x^m\sum_{i=1}^{m}\frac{(-1)^{i-1}x^{i(i-1)/2}}{(x)_{m-i}}.$$
 \end{thm}
 
The following lemma provides a useful interpretation of $H(x)$, which will be fundamental to prove Theorem \ref{gen_fun_U}.
 \begin{lemma} 
 \label{lemmaH=G+xK}
 Let $m \in \mathbb{N}$ and let $\Lambda_m$ be the set of partitions defined in Eq. \eqref{Lambda_m}.
Let $H(x)$ be the function defined in Eq. \eqref{H} and let $G(x)$ be as in Theorem $\ref{thmG}$. Let
$$K(x)\coloneqq \sum_{j=0}^{m-1}\frac{x^{(m-j)^2+m-1}(x)_{m-1}}{(x)_{m-j-1}^2(x)_j}.$$
Then
$H(x)=G(x)+xK(x)$, and
$$K(x)=\sum_{\lambda \in \Lambda_m}x^{|\lambda|}$$ coincides with the generating function for the number of partitions in $\Lambda_m$.

\begin{proof}
We have
\begin{equation*}
H(x)=\sum_{|\lambda|=m}\frac{(1-x^{2\lambda'_1})x^{\sum_i(\lambda'_i)^2}}{\prod_i(x)_{m_i(\lambda)}}=\sum_{j=0}^{m-1}\left(1-x^{2(m-j)}\right)\sum_{\substack{|\lambda|=m \\ \lambda'_1=m-j}}\frac{x^{\sum_i(\lambda'_i)^2}}{\prod_i(x)_{m_i(\lambda)}}.
\end{equation*}
Now, from the proof of \cite[Lemma~4.2]{SpigaAGL}, we have
\begin{equation*}
\sum_{\substack{|\lambda|=m \\ \lambda'_1=m-j}}\frac{x^{\sum_i(\lambda'_i)^2}}{\prod_i(x)_{m_i(\lambda)}}=\frac{x^{(m-j)^2+j}(x)_{m-1}}{(x)_{m-j-1}(x)_{m-j}(x)_j}.
\end{equation*}
Therefore, 
\begin{align*}
H(x)&=\sum_{j=0}^{m-1}(1+x^{m-j})(1-x^{m-j})\frac{x^{(m-j)^2+j}(x)_{m-1}}{(x)_{m-j-1}(x)_{m-j}(x)_j} \\
&=\sum_{j=0}^{m-1}(1+x^{m-j})\frac{x^{(m-j)^2+j}(x)_{m-1}}{(x)_{m-j-1}^2(x)_j}\\
&=\sum_{j=0}^{m-1}\frac{x^{(m-j)^2+j}(x)_{m-1}}{(x)_{m-j-1}^2(x)_j}+\sum_{j=0}^{m-1}\frac{x^{(m-j)^2+m}(x)_{m-1}}{(x)_{m-j-1}^2(x)_j} \\
&=G(x)+x\sum_{j=0}^{m-1}\frac{x^{(m-j)^2+m-1}(x)_{m-1}}{(x)_{m-j-1}^2(x)_j}=G(x)+xK(x),
\end{align*}
where the last two equalities follow from Theorem \ref{thmG} and the definition of $K(x)$. This proves the first part of the lemma.

\begin{figure}
\floatbox[{\capbeside\thisfloatsetup{capbesideposition={right,bottom},capbesidewidth=10 cm}}]{figure}[\FBwidth]
{\caption{\textbf{Example:} 
$\lambda=(8,7,7,4,4,3,3,1,1)$ has Durfee square 4, $\pi_1(\lambda)=(4,3,3)$, $\pi_2(\lambda)=(4,3,3,1,1)$ and it satisfies $\lambda_3>\lambda_4=4$.
}\label{fig:example_partition_lemma}}
{\includegraphics[]{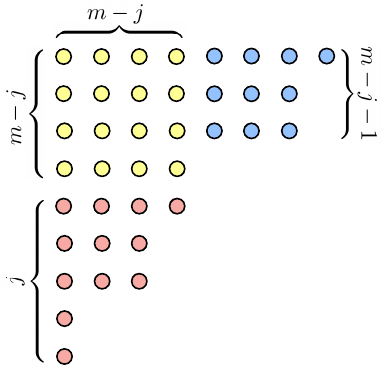}}
\end{figure}

To prove that $K(x)=\sum_{\lambda \in \Lambda_m}x^{|\lambda|}$, let us look at the summands of $K(x)$, which we rewrite as:
$$\frac{x^{(m-j)^2+m-1}(x)_{m-1}}{(x)_{m-j-1}^2(x)_j}=x^{(m-j)^2}\cdot \frac{x^j(x)_{m-1}}{(x)_{m-j-1}(x)_j}\cdot\frac{x^{m-j-1}}{(x)_{m-j-1}}.$$
From \cite[Theorem~3.1]{Andrews},
$$\frac{x^j(x)_{m-1}}{(x)_{m-j-1}(x)_j}$$
is the generating function for partitions $\pi_1$ into exactly $j$ parts, all having size at most $m-j$, and 
$$\frac{x^{m-j-1}}{(x)_{m-j-1}}$$
is the generating function for partitions $\pi_2$ with $m-j-1$ parts.
Therefore, since every partition $\lambda$ of $N$ is uniquely determined by the pair $(\pi_1(\lambda),\pi_2(\lambda))$, where $|\pi_1(\lambda)|+|\pi_2(\lambda)|=N-(m-j)^2$, $\pi_1(\lambda)$ is the partition below the Durfee square of $\lambda$ (which consists of exactly 
$j$ parts, each of size at most $m-j$), and $\pi_2(\lambda)$ is the partition to the right of its Durfee square (which has at most $m-j$ parts) (see Fig. \ref{fig:example_partition_lemma}),

$$x^{(m-j)^2}\cdot \frac{x^j(x)_{m-1}}{(x)_{m-j-1}(x)_j}\cdot\frac{x^{m-j-1}}{(x)_{m-j-1}}$$
is the generating function for partitions into $m$ parts, with Durfee square $m-j$, and with partition to the right of the Durfee square having exactly $m-j-1$ parts.

Therefore, $K(x)$ is the generating function for  partitions into $m$ parts, with Durfee square $m-j$ for some $j \in \{0,\dots,m-1\}$, and with the partition to the right of the Durfee square having exactly $m-j-1$ parts. Note that (see also Fig. \ref{fig:example_partition_lemma}) such partitions are exactly the partitions $\lambda=(\lambda_1,\dots,\lambda_m)$ into $m$ parts, such that either $\lambda_1=1$, or $\lambda_{k-1}>\lambda_k=k$ for some  $k \in \{2,\dots,m\}$, that is, the partitions in $\Lambda_m$. This concludes the proof.
\end{proof}
 \end{lemma}

We now prove Theorem \ref{p_main}\eqref{p_Unitary}, assuming Theorem \ref{gen_fun_U}, which will be established in Section \ref{Sec_gen_fun_U}.
\begin{proof}[Proof of Theorem $\ref{p_main}\eqref{p_Unitary}$]
From Lemma \ref{lemma_unip_U} we have
 $$u_m(q)=(-1)^m\sum_{|\lambda|=m}\frac{1-q^{-2\lambda'_1}}{(-q)^{\sum_{i}(\lambda'_i)^2}\prod_i(-1/q)_{m_i(\lambda)},}$$
 and from Eq. \eqref{u_m(q)=(-1)^mH(-1/q)} and Lemma \ref{lemmaH=G+xK} we have $u_m(q)=(-1)^mH(-1/q)$, where $H(x)=G(x)+xK(x)$. Combining Theorems \ref{gen_fun_U} and \ref{thmG}, we get
 \begin{align*}
 H(x)&=x^m\sum_{i=1}^{m}\frac{(-1)^{i-1}x^{i(i-1)/2}}{(x)_{m-i}}+
 \frac{x^{m+1}}{1-x}\sum_{i=1}^{m}\frac{(-1)^{i}x^{i(i-1)/2}(x^i-1)}{(x)_{m-i}} \\
 &=\frac{x^{m+1}}{1-x}\sum_{i=1}^{m}\frac{(-1)^{i-1}x^{i(i-1)/2}(x^{-1}-1)+(-1)^{i}x^{i(i+1)/2}+(-1)^{i-1}x^{i(i-1)/2}}{(x)_{m-i}} \\
 &=\frac{x^{m+1}}{1-x}\sum_{i=1}^m\frac{(-1)^{i}x^{i(i+1)/2}(1-x^{-(i+1)})}{(x)_{m-i}}.
 \end{align*}
 and therefore,
 \begin{equation*}
 u_m(q)=(-1)^mH(-1/q)=\frac{1}{q^m(q+1)}\sum_{i=1}^m\frac{(-1)^i((-q)^{i+1}-1)}{(-q)^{i(i+1)/2}(-1/q)_{m-i}}. \qedhere
  \end{equation*}
\end{proof} 
\subsection{Derangements: proof of Theorem \ref{main}\eqref{Unitary}}
Let us define
$$T_{\U}:=\sum_{m=0}^{\infty}\frac{y^m}{q^m(-1/q)_m}.$$

Recall the definitions of $U'_{\U}$ and $D'_{\U}$ in Eq. \eqref{D'_U}.

\begin{lemma}
\label{fact_D'_u}
We have
$$D'_{\U}=T^{-1}_{\U}(1-y)^{-1}U'_{\U}.$$
\begin{proof}
Recall that for a unipotent element $a \in \U_m(q)$, the partition $\lambda$ associated with $z-1$ satisfies $\lambda'_1=\dim \ker(a-1)$. It follows from the definition of the cycle index $Z_{\U}$ 
 that $d'_m(q)$ is the coefficient of $y^m$ in $Z_{\U}$ when we assign the variables $x_{\phi,\lambda}$ and $x_{\widetilde{\phi},\lambda}$ the value 1 when $\phi \neq z-1$, and set $x_{z-1,\lambda}=q^{-2\lambda'_1}.$
 Moreover, observe that if all variables $x_{\phi,\lambda}$ and $x_{\widetilde{\phi},\lambda}$ are set to $1$ in $Z_{\U}$, then we obtain $(1-y)^{-1}$. Similarly, note that $u'_m(q)$ is the coefficient of $y^m$ in the factor 
  $$\sum_{\lambda}x_{z-1,\lambda}\frac{(-y)^{|\lambda|}}{(-q)^{\sum_{i}(\lambda'_i)^2}\prod_i(-1/q)_{m_i(\lambda)}}$$
of $Z_{\U}$, when we set the variables $x_{z-1,\lambda}$ equal to $q^{-2\lambda'_1}$. If we instead set these variables $x_{z-1,\lambda}$ equal to 1,
the coefficient of $y^m$ in $\sum_{\lambda}\frac{(-y)^{|\lambda|}}{(-q)^{\sum_{i}(\lambda'_i)^2}\prod_i(-1/q)_{m_i(\lambda)}}$ is equal to the proportion of unipotent elements in $\U_m(q)$, which, by Lemma \ref{steinberg}, is $\frac{1}{q^m(-1/q)_m}$.
Therefore, 
 $$\sum_{\lambda}\frac{(-y)^{|\lambda|}}{(-q)^{\sum_{i}(\lambda'_i)^2}\prod_i(-1/q)_{m_i(\lambda)}}=
 \sum_{m=0}^{\infty}\frac{y^m}{q^m(-1/q)_m}=T_{\U}.$$
Using the factorisation of $Z_{\U}$ in Eq. \eqref{fact_Z_u}, these observations together imply that
 \begin{equation*}
 (1-y)^{-1}=D'_{\U}\cdot U'^{-1}_{\U} \cdot T_{\U}. \qedhere 
\end{equation*}
\end{proof}
\end{lemma}

We are ready to prove Theorem \ref{main}\eqref{Unitary} (assuming Theorem \ref{gen_fun_U}).
\begin{proof}[Proof of Theorem  $\ref{main}\eqref{Unitary}$]
Applying Theorem \ref{p_main}\eqref{p_Unitary}, we first obtain:
\begin{align*}
u'_m(q)& =\frac{1}{q^m(-1/q)_m}-u_m(q) \\
& =\frac{1}{q^m(-1/q)_m}-\frac{1}{q^m(q+1)}\sum_{i=1}^m\frac{(-1)^{i}((-q)^{i+1}-1)}{(-q)^{i(i+1)/2}(-1/q)_{m-i}}\\
& = \frac{1}{q^m(q+1)}\sum_{i=0}^{m}\frac{(-1)^{i}(1-(-q)^{i+1})}{(-q)^{i(i+1)/2}(-1/q)_{m-i}}.
\end{align*}
We now observe that the following factorisation holds for $U'_{\U}$:
\begin{align}
    \nonumber
U'_{\U}=\sum_{m=0}^{\infty}u'_m(q)y^m&=\sum_{m=0}^{\infty}\frac{1}{q+1}\sum_{i=0}^{m}\frac{(-1)^i(1-(-q)^{i+1})}{(-q)^{m-i}(-q)^{i(i+3)/2}(-1/q)_{m-i}}(-y)^m\\
\nonumber
    &=\sum_{m=0}^{\infty}\frac{1}{q+1}\sum_{i=0}^{m}\frac{(-1)^i(1-(-q)^{i+1})}{(-q)^{i(i+3)/2}}(-y)^i\cdot\frac{1}{(-q)^{m-i}(-1/q)_{m-i}}(-y)^{m-i}\\
    \label{fattorizz}
    &=\overline{D}_{\U}\cdot T_{ \U},
\end{align}
where $\overline{D}_{\U}$ is defined as
$$\overline{D}_{\U}:=\frac{1}{q+1}\sum_{i=0}^{\infty}\frac{1-(-q)^{i+1}}{(-q)^{i(i+3)/2}}y^i.$$
By Eq. \eqref{fattorizz} and Lemma \ref{fact_D'_u}, we have 
\begin{align*}
D'_{\U}&=T^{-1}_{\U}\cdot U'_{\U} \cdot (1-y)^{-1}= \overline{D}_{\U} \cdot (1-y)^{-1} \\
&=\frac{1}{q+1}\sum_{i=0}^{\infty}\frac{1-(-q)^{i+1}}{(-q)^{i(i+3)/2}}y^i\cdot \sum_{j=0}^{\infty}y^j=\sum_{m=0}^{\infty}\frac{1}{q+1}\sum_{i=0}^{m}\frac{1-(-q)^{i+1}}{(-q)^{i(i+3)/2}}y^m.
\end{align*}
Hence,
$$d'_m(q)=\frac{1}{q+1}\sum_{i=0}^{m}\frac{1-(-q)^{i+1}}{(-q)^{i(i+3)/2}}$$
and
$$d_m(q)=1-d'_m(q)=\frac{1}{q+1}\sum_{i=1}^{m}\frac{(-q)^{i+1}-1}{(-q)^{i(i+3)/2}}=\frac{1}{q+1}\left(1-\frac{1}{(-q)^{m(m+3)/2}}\right),$$
where the last equality can be easily verified by induction on $m$.
\end{proof}

\subsection{Proof of Theorem \ref{gen_fun_U}}
\label{Sec_gen_fun_U}
Let $m\ge1$. In the introduction (see Eq. \eqref{Lambda_m}), we defined $\Lambda_m$ as the set of partitions 
$\lambda=(\lambda_1,\dots,\lambda_m)$ 
 such that either $\lambda_1=1$ or $\lambda_{k-1}>\lambda_k=k$, for some $k \in \{2,\dots,m\}$.
 Considering any number $m$ of parts for such partitions, we define
 \begin{equation*}
\Lambda\coloneqq \bigcup_{m\ge1}\Lambda_m.
 \end{equation*}
 For example, the partitions $(6,5,3,3,2)$, $(9,6,5,4,3,1,1)$ and $(1,1,1)$ belong to $\Lambda$, and all the partitions in $\Lambda_1,$ $\Lambda_2$ and $\Lambda_3$ are listed in the first column of Table \ref{exgenfun}.

As observed in the proof of Lemma \ref{lemmaH=G+xK}, $\lambda \in \Lambda_m$ if and only if it has Durfee square $k$ (for some $k \in \{1,\dots,m\}$)  and the partition to the right of the Durfee square has exactly $k-1$ parts (see also Fig. \ref{fig:example_partition_lemma}). Partitions satisfying the condition $\lambda_k=k$ for some $k$ were studied by Spiga in \cite{SpigaAGL}, in relation to the proportion of derangements in $\AGL_m(q)$. In \cite{BlechKnopf}, Blecher and Knopfmacher refer to the condition $\lambda_k=k$ as a \textit{fixed point}, extending the concept of fixed points from permutation group theory to integer partitions. In addition to the fixed point condition, the partitions in $\Lambda$ impose the extra requirement that $\lambda_{k-1}>k$, which can be rephrased by saying that the conjugate of the partition to the right of the Durfee square must have first part equal to $k-1$ (see also \cite{HopkSell} and \cite{Hopk}).

Theorem \ref{gen_fun_U} states that the generating function for the number of partitions in $\Lambda_m$ is
$$\frac{x^m}{1-x}\sum_{i=1}^{m}\frac{(-1)^{i}x^{i(i-1)/2}(x^i-1)}{(x)_{m-i}}.$$

To prove Theorem  \ref{gen_fun_U}, we first establish the following auxiliary lemma. The author would like to thank Fedor Petrov for providing the proof of this lemma, in response to a question posed on MathOverflow, see \cite{Fedor}.
\begin{lemma}
\label{|A|=|B|}
Let $a$ and $b$ be non-negative integers. The following two sets have the same cardinality:
\begin{align*}
\mathcal{A}(a,b)\coloneqq&\{(\lambda,\mu) \mid \lambda,\mu \textup{ partitions}, \pt(\lambda)=b, |\lambda|+|\mu|=a, \lambda_1=\mu_1+1\},\\
\mathcal{B}(a,b)\coloneqq&\{(\lambda,\mu) \mid \lambda,\mu \textup{ partitions}, \pt(\lambda)=b, |\lambda|+|\mu|=a, \lambda \in \Lambda, \mu_1=\mu_2 \},
\end{align*}
where in the set $\mathcal{B}(a,b)$,  $\mu$ can be the empty partition.
\end{lemma}

Before starting the proof, let us clarify the definitions of the sets $\mathcal{A}(a,b)$ and $\mathcal{B}(a,b)$ with an example. Let $a=9$ and $b=4$. Then
\begin{align*}
\mathcal{A}(9,4)=\{& ((2,1,1,1),(1,1,1,1)), & \mathcal{B}(9,4)=\{&((1,1,1,1),(1,1,1,1,1)),\\
& ((3,1,1,1),(2,1)), & &((1,1,1,1),(2,2,1)), \\
& ((3,2,1,1),(2)),& &((3,2,1,1),(1,1)), \\
& ((2,2,1,1),(1,1,1)) & & ((3,2,2,2),\emptyset),\\
& ((2,2,2,1),(1,1))& &((4,2,2,1),\emptyset), \\
& ((2,2,2,2),(1))\} & &((5,2,1,1), \emptyset)\}.
\end{align*}

\begin{proof}[Proof of Lemma $\ref{|A|=|B|}$]
Let us define a bijection 
$$
\Phi(a,b):\mathcal{A}(a,b) \longrightarrow \mathcal{B}(a,b), \quad (\lambda,\mu)\longmapsto (\lambda^*,\mu^*).
$$
Figure \ref{fig:bijection} provides an auxiliary illustration for the proof.

\begin{figure}
{\caption{\textbf{Graphical representation of $\Phi(a,b)$}\label{fig:bijection}}}
\includegraphics[width=\textwidth]{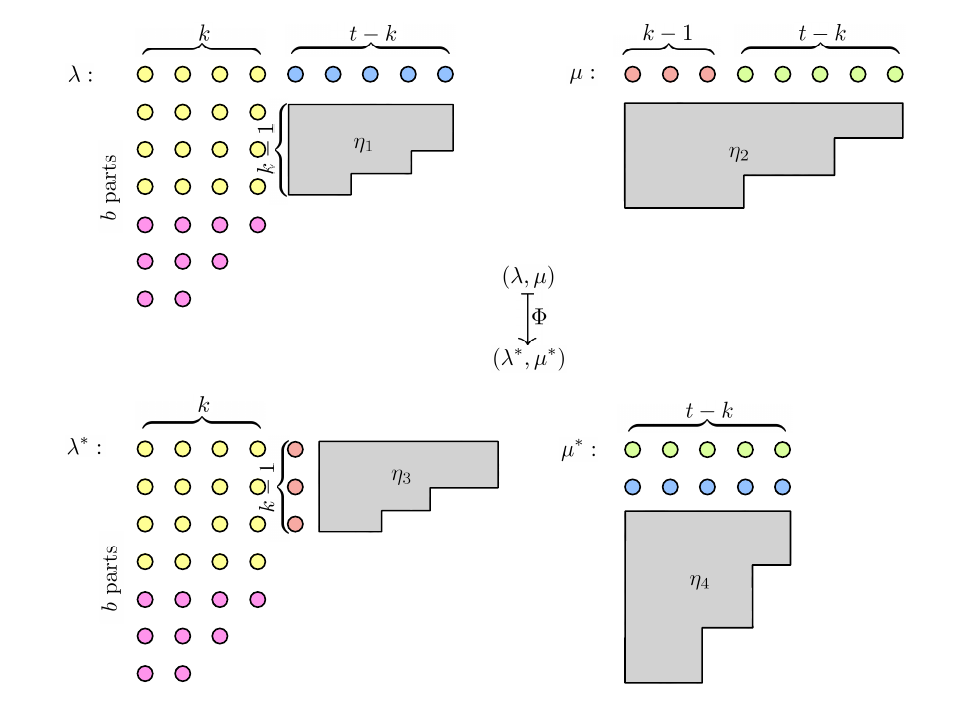}
\end{figure}

To define $\Phi(a,b)$, we first introduce the following map. For non-negative integers $u$ and $v$, let $\phi_{u,v}$ denote a bijection between pairs of partitions $(\eta_1, \eta_2) \mapsto (\eta_3, \eta_4)$, where:
\begin{itemize}
    \item $\eta_1$ (meaning its Ferrers diagram) fits in the rectangle $u \times v$,
    \item $\eta_2$ fits in the strip $(u+v) \times \infty$,
    \item $\eta_3$ fits in the strip $\infty \times v$,
    \item $\eta_4$ fits in the strip $u \times \infty$,
    \item $|\eta_1|+|\eta_2|=|\eta_3|+|\eta_4|$.
\end{itemize}
The existence of such a bijection is guaranteed by the Gaussian binomial identity, which states that
$$\binom{u+v}{u}_q=\frac{(q)_{u+v}}{(q)_u(q)_v},$$
where the $q$-binomial coefficient $\binom{u+v}{u}_q$ is the generating function for partitions whose Ferrers diagram fits inside a $u \times v$ rectangle, and $\frac{1}{(q)_u}$ is the generating function for partitions into at most $u$ parts (see, for instance, \cite{Andrews} or any standard textbook on partitions).

Now, let $(\lambda,\mu) \in \mathcal{A}(a,b)$, with $\lambda_1=t$ (and hence $\mu_1=t-1)$, and assume that the Durfee square of $\lambda$ is $k$, where $t$ and $k$ are positive integers.
Let $\pi_1(\lambda)$ be the partition to the right of the Durfee square of $\lambda$, let $\pi_2(\lambda)$ be the partition below the Durfee square. Note that $\lambda$ is uniquely determined by the size of its Durfee square together with $\pi_1(\lambda)$ and $\pi_2(\lambda)$.

We now define $\Phi(\lambda,\mu)\coloneqq(\lambda^*, \mu^*)$ as follows.
\begin{itemize}
    \item $\Phi$ preserves the size of the Durfee square of $\lambda$, that is $\lambda^*$ has the same Durfee square size as $\lambda$.
    \item $\Phi$ preserves the partition below the Durfee square of $\lambda$, i.e.  $\pi_2(\lambda)=\pi_2(\lambda^*)$.
    \item  Let $\pt(\pi_1(\lambda))=n$ and $\pt(\mu)=m$. Define
    $$\eta_1 \coloneqq((\pi_1(\lambda))_2,\dots,(\pi_1(\lambda))_n),$$ that is, $\eta_1$ consists of all the parts of $\pi_1(\lambda)$, except the first one, and 
    $$\eta_2\coloneqq(\mu_2,\dots,\mu_m),$$
    and let $(\eta_3,\eta_4)\coloneqq \phi_{t-k,k-1}(\eta_1,\eta_2)$. Set
    \begin{align*}
    \pi_1(\lambda^*)&\coloneqq ((\eta_3)_1+1,\dots,(\eta_3)_{k-1}+1)\quad \text{and}\\
    \mu{^*}&\coloneqq (t-k,t-k,\eta_4). 
     \end{align*}
    \end{itemize}

Then, $\Phi(a,b)$ is indeed a bijection: the construction reduces to check fixed values of $t$ and $k$, and the conclusion follows from the fact that
$\phi_{t-k,k-1}$ is itself a bijection.
\end{proof}

We are now ready to prove Theorem \ref{gen_fun_U}. The structure of the proof follows that of \cite[Theorem 1.3]{SpigaAGL}.
\begin{proof}[Proof of Theorem $\ref{gen_fun_U}$]
Let
$$F(x,y) \coloneqq \sum_{\lambda \in \Lambda }y^{\pt(\lambda)}x^{|\lambda|}$$
denote the generating function for the number of partitions in $\Lambda$, encoded by size and by number of parts. Define
$$\bar{F}(x,y)\coloneqq \sum_{m=0}^{\infty}\left( \frac{x^m}{1-x}\sum_{i=1}^{m}\frac{(-1)^{i}x^{i(i-1)/2}(x^i-1)}{(x)_{m-i}}\right)y^m.$$
To prove the claim, we need to show that $F(x,y)=\bar{F}(x,y)$. 
Let
\begin{align*}
F_1(x,y)\coloneqq\sum_{a=0}^{\infty}\frac{x^a}{(x)_a}y^a, &&F_2(x,y)\coloneqq\sum_{b=0}^{\infty}(x^b-1)x^{\frac{b(b+1)}{2}}(-y)^b,
\end{align*}
and observe that
\begin{equation}
\label{F_bar}
\bar{F}(x,y)=\frac{1}{1-x}F_1(x,y)F_2(x,y).
\end{equation}
Recalling that $x^a/(x)_a$ is the generating function for the number of partitions into $a$ parts, we get that
$$F_1(x,y)=\sum_{\lambda}x^{|\lambda|}y^{\pt(\lambda)}$$
is the generating function for the number of partitions, encoded by size and by number of parts. From \cite[p.~16]{Andrews} (or from any standard textbook on partitions), it follows that
$$
F_1(x,y)=\prod_{n=1}^{\infty}(1-yx^n)^{-1}.
$$
Now, we turn our attention to 
\begin{equation}
    \label{F_2}
F_2(x,y)=\sum_{b=0}^{\infty}(x^b-1)x^{\frac{b(b+1)}{2}}(-y)^b=\sum_{b=0}^{\infty}x^{\frac{b(b+3)}{2}}(-y)^b-\sum_{b=0}^{\infty}x^{\frac{b(b+1)}{2}}(-y)^b.
\end{equation}
In view of Eq. \eqref{F_2}, let us define
\begin{align*}
F_2^1(x,y)\coloneqq\sum_{b=-\infty}^{\infty}x^{\frac{b(b+3)}{2}}(-y)^b &&\text{and} &&F_2^2(x,y) \coloneqq\sum_{b=-\infty}^{\infty}x^{\frac{b(b+1)}{2}}(-y)^b.
\end{align*}
We now require Jacobi's triple product identity, which states that
\begin{equation}
\label{Jacobi_triple}
\sum_{j=-\infty}^{\infty}q^{j^2}z^j=\prod_{j=1}^{\infty}(1+zq^{2j-1})(1+z^{-1}q^{2j-1})(1-q^{2j}).
\end{equation}
We apply Eq. \eqref{Jacobi_triple} to both $F_2^1(x,y)$ and $F_2^2(x,y)$, obtaining:
\begin{align*}
F_2^1(x,y)&=\sum_{b=-\infty}^{\infty}x^{\frac{b(b+3)}{2}}(-y)^b=\sum_{b=-\infty}^{\infty}(x^{\frac{1}{2}})^{b^2}(-yx^{\frac{3}{2}})^b \\
&=\prod_{b=1}^{\infty}(1-yx^{\frac{3}{2}}x^{\frac{2b-1}{2}})(1-y^{-1}x^{-\frac{3}{2}}x^{\frac{2b-1}{2}})(1-x^b)\\
&=\prod_{b=1}^{\infty}(1-yx^{b+1})(1-y^{-1}x^{b-2})(1-x^b),
\end{align*}
and
\begin{align*}
F_2^2(x,y)&=\sum_{b=-\infty}^{\infty}x^{\frac{b(b+1)}{2}}(-y)^b=\sum_{b=-\infty}^{\infty}(x^{\frac{1}{2}})^{b^2}(-yx^{\frac{1}{2}})^b \\
&=\prod_{b=1}^{\infty}(1-yx^{\frac{1}{2}}x^{\frac{2b-1}{2}})(1-y^{-1}x^{-\frac{1}{2}}x^{\frac{2b-1}{2}})(1-x^b)\\
&=\prod_{b=1}^{\infty}(1-yx^{b})(1-y^{-1}x^{b-1})(1-x^b).
\end{align*}
Note that $F_2^1(x,y)=-\frac{1}{xy}F_2^2(x,y)$. Therefore,
\begin{equation}
\label{jacobi}
F_2^1(x,y)-F_2^2(x,y)=-\frac{xy+1}{xy}\prod_{b=1}^{\infty}(1-yx^b)(1-y^{-1}x^{b-1})(1-x^b).
\end{equation}
The function we are interested in is $F_2(x,y)$, which is obtained by restricting the infinite sum 
$$F_2^1(x,y)-F_2^2(x,y)=\sum_{b=-\infty}^{\infty}(-y)^bx^{b(b+3)/2}-(-y)^bx^{b(b+1)/2},$$ 
to the non-negative indices $b \ge 0$. In view of Eq. \eqref{jacobi}, let $F_3(x,y)$ be the part of the product $-\frac{xy+1}{xy}\prod_{b=1}^{\infty}(1-yx^b)(1-y^{-1}x^{b-1})$ corresponding to the non-negative powers of $y$.
Then
$$F_2(x,y)=F_3(x,y)\prod_{n=1}^{\infty}(1-x^n)$$
and hence, using Eq. \eqref{F_bar}, we have
$$\bar{F}(x,y)=\frac{1}{1-x}F_1(x,y)F_2(x,y)=\prod_{n=2}^{\infty}(1-x^n)\prod_{n=1}^{\infty}(1-yx^n)^{-1}F_3(x,y).$$
Recall that we want to prove that $F(x,y)=\bar{F}(x,y)$, and to do this, it suffices to verify that $F(x,y)$ satisfies the same identity as $\bar{F}(x,y)$, that is
\begin{equation}
\label{identity_F}
F(x,y)\prod_{n=2}^{\infty}(1-x^n)^{-1}=\prod_{n=1}^{\infty}(1-yx^{n})^{-1}F_3(x,y).
\end{equation}
To conclude the proof, we proceed to show that for every pair of non-negative integers $a$ and $b$, the coefficients of $x^ay^b$ on both sides of Eq. \eqref{identity_F} coincide.

Note that $\prod_{n=2}^{\infty}(1-x^n)^{-1}$ is the generating function for partitions with no part of length 1, hence, dually, it is the generating function for partitions $\mu$ satisfying $\mu_1=\mu_2$. Since $F(x,y)$ is the generating function of $\Lambda$, the coefficient of $x^ay^b$ on the left hand side of Eq. \eqref{identity_F} equals the cardinality of the set:
\begin{equation}
\label{B(a,b)}
\mathcal{B}(a,b) =\{(\lambda,\mu) \mid \lambda, \mu \text{ partitions }, \pt(\lambda)=b, |\lambda|+|\mu|=a, \lambda \in \Lambda, \mu_1=\mu_2 \}.
\end{equation}
We now turn our attention to the right hand side of Eq. \eqref{identity_F}. Since $  F_3(x,y)$ is the partial sum  of the infinite product $-\left(1+\frac{1}{xy}\right)\prod_{n=1}^{\infty}(1-yx^n)(1-y^{-1}x^{n-1})$,
corresponding to the non-negative powers of $y$, we consider instead the expanded product

\begin{equation}
\label{-A-A'}
-\left(1+\frac{1}{xy}\right)\prod_{n=1}^{\infty}(1+yx^n+y^2x^{2n}+y^3x^{3n}+\cdots)\prod_{n=1}^{\infty}(1-yx^n)(1-y^{-1}x^{n-1})
=-A(x,y)-A'(x,y)
\end{equation}
where we have defined
\begin{align*}
A(x,y)&\coloneqq \prod_{n=1}^{\infty}(1+yx^n+y^2x^{2n}+y^3x^{3n}+\cdots)\prod_{n=1}^{\infty}(1-yx^n)\prod_{n=1}^{\infty}(1-y^{-1}x^{n-1}), \\
\intertext{and}
A'(x,y)&\coloneqq\frac{1}{xy}A(x,y).
\end{align*}
Let us compute the contribution to $x^ay^b$ from each factor in $A(x,y)$ (respectively $A'(x,y)$), taking into account that the contributions from the infinite product $\prod_{n=1}^{\infty}(1-yx^n)(1-y^{-1}x^{n-1})$ (respectively $\frac{1}{xy}\prod_{n=1}^{\infty}(1-yx^n)(1-y^{-1}x^{n-1})$) must come from a non-negative power of $y$. This means that:
\begin{enumerate}
    \item in $A(x,y)$ we must take at least as many $y$s from the second infinite product as $y^{-1}$ from the third infinite product;
    \item in $A'(x,y)$, since there is a factor $(xy)^{-1}$, we are required to take at least one more $y$ from the second infinite product than the number of $y^{-1}$ from the third infinite product.
\end{enumerate}
Let us formalise what we have just described. 
We define an  \textit{$(a,b)$-choice} in $A(x,y)$ (resp. in $A'(x,y)$), or simply, a \textit{choice}, to be a contribution to the summand $x^ay^b$ in $A(x,y)$ (resp. in $A'(x,y)$). Precisely, an $(a,b)$-choice is uniquely determined by selecting terms $$y^{v_1}x^{v_1c_1},\dots,y^{v_{l_1}}x^{v_{l_1}c_{l_1}}$$ from the first infinite product in $A(x,y)$ (resp. $A'(x,y)$) (and 1 from all other factors in the first  product), terms $$-yx^{a_1},\dots,-yx^{a_{l_2}}$$ from the second infinite product (and 1 from all other factors in the second product), and terms $$-y^{-1}x^{b_1},\dots,-y^{-1}x^{b_{l_3}}$$ from the third infinite product (and 1 from all other factors in the third product). 
The following conditions must be satisfied:
\begin{enumerate}
    \item for $A(x,y)$: $l_2 \ge l_3$, $\sum_k v_k+l_2-l_3=b$ and $\sum_k v_k c_k+\sum_ia_i+\sum_j b_j=a$;
    \item for $A'(x,y)$: $l_2 > l_3$, $\sum_k v_k+l_2-l_3=b+1$ and $\sum_k v_k c_k+\sum_ia_i+\sum_j b_j=a+1$.
\end{enumerate}

Determining the coefficient of $x^ay^b$ in $A(x,y)$ (resp. in $A'(x,y)$) amounts to summing over all $(a,b)$-choices in $A(x,y)$ (resp. in $A'(x,y)$). 
When performing this summation, many choices cancel out pairwise. We now describe how this cancellation occurs.

Given a fixed $(a,b)$-choice \textbf{c} in $A(x,y)$ (resp. $A'(x,y)$), let $n$ be the maximal index for which either $y^{v}x^{vn}$ (for some $v \ge 1$), or $-yx^n$ appears in \textbf{c}. There are three possible cases.
\begin{itemize}
    \item If both $y^{v}x^{vn}$ and $-yx^n$ appear in \textbf{c}, then \textbf{c} cancels with the choice obtained by replacing $y^{v}x^{vn}$ with $y^{v+1}x^{(v+1)n}$ and removing $-yx^n$  (this corresponds to taking 1 in the factor $1-yx^n$ of the second infinite product).
    \item If $y^vx^{vn}$ appears in \textbf{c} but $-yx^n$ does not, then \textbf{c} cancels with the choice obtained by replacing $y^vx^{vn}$ with $y^{v-1}x^{(v-1)n}$ and selecting $-yx^n$ in the second infinite product.
    \item Finally, if $-yx^n$ appears in \textbf{c} but $y^vx^{vn}$ does not, then this choice cancels with the choice where $-yx^n$ is replaced by $yx^n$ (this corresponds to selecting $yx^n$ in the factor $1+yx^n+y^2x^{2n}+\dots$ in the first infinite product, and selecting $1$ in the factor $1-yx^n$ of the second infinite product).
\end{itemize}

Observe that we cannot always perform such cancellations: when a term $-yx^n$ is removed from \textbf{c}, we must ensure that the number of terms of the form $-yx^m$ in the cancelling choice is at least, for $A(x,y)$, (resp. strictly bigger than, for $A'(x,y)$) the number of terms of the form $-y^{-1}x^{m'}$ Recall that this corresponds the conditions $l_2 \ge l_3$ for $A(x,y)$, and $l_2 > l_3$ for $A'(x,y)$.
Moreover, note that this cancellation process is an involution: applying the same procedure to the cancelling choice recovers the original choice \textbf{c}.

Therefore, in $A(x,y)$, the $(a,b)$-choices that remain (i.e., those that do not cancel) are exactly those where the number of $y$-terms equals the number of $y^{-1}$-terms, and $c_{l_1} \le a_{l_2}$ (recall that these quantities appeared in the definition of an $(a,b)$-choice above). Similarly, in $A'(x,y)$, the remaining $(a,b)$-choices are those where the number of $y$s is exactly one more than the number of $y^{-1}$s, and $c_{l_1} \le a_{l_2}$. Furthermore, observe that all remaining choices in $A(x,y)$ have positive sign, while all those in $A'(x,y)$ carry negative sign.

Given this analysis, determining the coefficient of $x^ay^b$ in $A(x,y)$ (respectively $A'(x,y))$  amounts to counting the number of the following possibilities.
\begin{enumerate}[label=(\arabic*)]
    \item For $A(x,y)$, choose: 
    \begin{enumerate}[label=(\roman*)]
        \item a non-negative integer $j$, and integers $0<c_1 < \cdots < c_j$ and $0 <v_1< \cdots <v_j$ (for the terms $y^{v_i}x^{v_ic_i}$ in the first infinite product),
        \item $k$ positive integers $0<a_1<a_2<\cdots<a_k$ (for the terms $yx^{a_i}$ in the second product),
        \item $k$ non-negative integers $0 \le b_1 < b_2 < \cdots < b_k$ (for the terms $y^{-1}x^{b{_i}}$ in the third product),
    \end{enumerate}
satisfying
$$c_j \le a_k,\text{ } \sum_{i=1}^jv_i=b,\text{ } \sum_{i=1}^{j}v_ic_i+\sum_{i=1}^ka_i+b_i=a.$$
\item For $A'(x,y)$, choose:
 \begin{enumerate}[label=(\roman*)]
        \item a non-negative integer $j$, and integers $0<c_1 < \cdots < c_j$ and $0 <v_1< \cdots <v_j$ (for the terms $y^{v_i}x^{v_ic_i}$ in the first infinite product),
        \item $k$ positive integers $0<a_1<a_2<\cdots<a_k$ (for the terms $yx^{a_i}$ in the second product),
        \item $k-1$ non-negative integers $0 \le b_1 < b_2 < \cdots < b_{k-1}$ (for the terms $y^{-1}x^{b{_i}}$ in the third product),
    \end{enumerate}
satisfying
$$c_j \le a_k,\text{ }\sum_{i=1}^jv_i=b,\text{ }\sum_{i=1}^{j}v_ic_i+\sum_{i=1}^ka_i+\sum_{i=1}^{k-1}b_i=a+1.$$
\end{enumerate}
Then, to obtain the coefficient of $x^ay^b$ on the right hand side of Eq. \eqref{identity_F}, subtract the number of choices 
in (2) from the number of choices in (1).

Note that, for both $A(x)$ and $A'(x)$, the choice at point (i), with the condition that $c_j \le a_k$, corresponds to selecting a partition $\lambda$ with maximal part of size at most $a_k$.

Now, for $A(x)$, the choices at point (i) and (ii), correspond to selecting  a partition $\mu$ with maximal part equal to $a_k$. Indeed, for such a partition, let $k$ be the size of its Durfee square, then $a_1,\dots,a_k$ correspond to the part of the Ferrers diagram above the main diagonal, and $b_1,...,b_k$ correspond to the part below it (as in the first example of Fig. \ref{fig:examples_mu}).

On the other hand, for $A'(x)$, the choices at point (i) and (ii) correspond to selecting  a partition $\mu$ with maximal part $a_k$ and $\mu_2 < \mu_1$: indeed, for such a partition, let $\mu_1=a_1$, and if we write $\overline{\mu}=(\mu_2,\dots,\mu_r)$, then, as above, $a_2,\dots,a_k$ correspond to the part of the Ferrers diagram of $\overline{\mu}$ above the diagonal of its Durfee square, and $b_1,\dots,b_k$ to the part below.
To clarify these constructions, see the examples in Fig. \ref{fig:examples_mu}. Note in particular that the motivation for the second construction relies on the fact that the number of $b_i$s is one less than the number of $a_j$s, and that $b_1$ may be 0 (as in the second example of Fig. \ref{fig:examples_mu}).
\begin{figure}
\includegraphics[width=\textwidth]{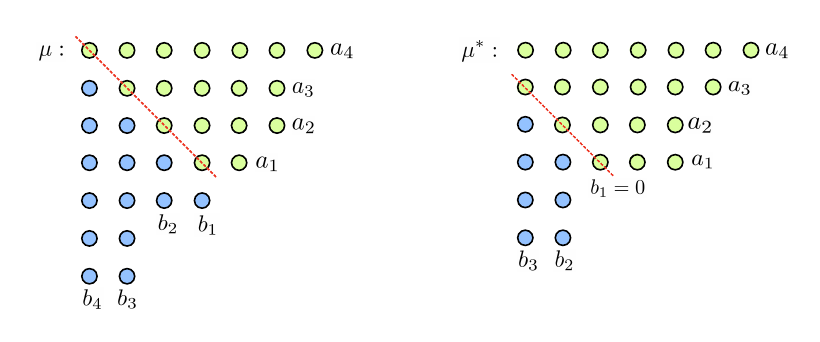}
\caption{\label{fig:examples_mu}}
\end{figure}
Summing up, we can rephrase points (1) and (2) above as follows:
\begin{enumerate}[label=(\arabic*)]
    \item for $A(x,y)$ choose
    \begin{enumerate}[label=(\roman*)]
        \item a partition $\lambda$ with maximal part $a_k$,
        \item a partition $\mu$ with maximal part at most $a_k$,
    \end{enumerate}
    satisfying $\pt(\lambda)=b$ and $|\lambda|+|\mu|=a$;
     \item for $A'(x,y)$ choose
    \begin{enumerate}[label=(\roman*)]
        \item a partition $\lambda$ with maximal part $a_k$,
        \item a partition $\mu$ with maximal part at most $a_k$ and $\mu_1 >\mu_2$,
    \end{enumerate}
    satisfying $\pt(\lambda)=b$ and $|\lambda|+|\mu|=a+1$.
\end{enumerate}
Defining 
\begin{align*}
\mathcal{E}(a,b)&\coloneqq\{(\lambda,\mu) \mid \pt(\lambda)=b, |\lambda|+|\mu|=a, \max(\lambda) \le \max(\mu)\},\\
\mathcal{F}(a,b)&\coloneqq\{(\lambda,\mu) \mid \pt(\lambda)=b, |\lambda|+|\mu|=a+1, \max(\lambda) \le \max(\mu), \mu_2 <\mu_1\}
\end{align*}
and recalling Eq. \eqref{-A-A'}, we obtain that the coefficient of $x^ay^b$ on the right hand side of Eq. \eqref{identity_F} coincides with the difference $|\mathcal{E}(a,b)|-|\mathcal{F}(a,b)|.$ Note that there is a bijective correspondence between $\mathcal{F}(a,b)$ and
$$\mathcal{G}(a,b)\coloneqq
\{(\lambda,\mu) \mid \pt(\lambda)=b, |\lambda|+|\mu|=a, \max(\lambda) \le \max(\mu)+1\},
$$
obtained by subtracting -1 to the part $\mu_1$ of $\mu$ in $\mathcal{F}(a,b)$.
Therefore, $|\mathcal{E}(a,b)|-|\mathcal{F}(a,b)|=|\mathcal{E}(a,b)|-|\mathcal{G}(a,b)|=|\mathcal{A}(a,b)|$, where 
$$\mathcal{A}(a,b)=\{(\lambda,\mu) \mid \lambda,\mu \textup{ partitions}, \pt(\lambda)=b, |\lambda|+|\mu|=a, \lambda_1=\mu_1+1\}.$$
Since we obtained that the coefficient of $x^ay^b$ in the left hand side of Eq. \eqref{identity_F} is the cardinality of the set $\mathcal{B}(a,b)$ defined in Eq. \eqref{B(a,b)}, applying Lemma \ref{|A|=|B|} we can finally conclude.
\end{proof}

\section{Affine symplectic groups}
\label{sec_symplectic}
The symplectic group $\Sp_{2m}(q)$ is the subgroup of $\GL_{2m}(q)$ that preserves a non-degenerate symplectic form on the vector space $\mathbb{F}_q^{2m}$.

Before describing the cycle index of the symplectic groups, let us recall Wall's combinatorial description of its conjugacy classes. We describe in detail the odd-characteristic case, following Fulman's treatment of the subject in \cite[Sec. 4.2]{Fulman_cycle_index}. 

Assume $q$ odd. Given a polynomial $\phi \in \mathbb{F}_q[z]$, with non-zero constant term, define the polynomial $\bar{\phi}$ by
\begin{equation}
\label{bar_phi}
\bar{\phi}(z)=\phi(0)^{-1}z^{\deg(\phi)}\phi(z^{-1}).
\end{equation}
Wall \cite{Wall} proved that a conjugacy class in $\Sp_{2m}(q)$ is parametrised by the following combinatorial data. To each monic, non-constant, irreducible polynomial $\phi \neq z\pm1$ associate a partition $\lambda_{\phi}$ of some non-negative integer $|\lambda_{\phi}|$. To $\phi$ equal to $z-1$ or $z+1$ associate a \textit{symplectic signed partition} $\lambda_{\phi}^{\pm}$, which means a partition of some natural number $|\lambda^{\pm}_{\phi}|$ where all odd parts occur with even multiplicity, together with a choice of sign for the set of parts of size $i$, for each even $i>0$, see Fig. \ref{fig:symplectic_signed_partition} for an example.
These data correspond to a conjugacy class in $\Sp_{2m}(q)$ if and only if 
\begin{enumerate}[label=(\roman*)]
    \item $|\lambda_z|=0$,
    \item $\lambda_{\phi}=\lambda_{\bar{\phi}}$,
    \item $\sum_{\phi=z\pm1}|\lambda_{\phi}^{\pm}|+\sum_{\phi \neq z \pm1}|\lambda_{\phi}|\deg (\phi)=2m$.
\end{enumerate}

\begin{figure}
\floatbox[{\capbeside\thisfloatsetup{capbesideposition={right,bottom},capbesidewidth= 11cm}}]{figure}[\FBwidth]
{\caption{\textbf{Example of a symplectic signed partition:} here, the $+$ corresponds to the part of size 8 and the $-$ corresponds to the parts of size 4 and 2}\label{fig:symplectic_signed_partition}}
{\includegraphics[]{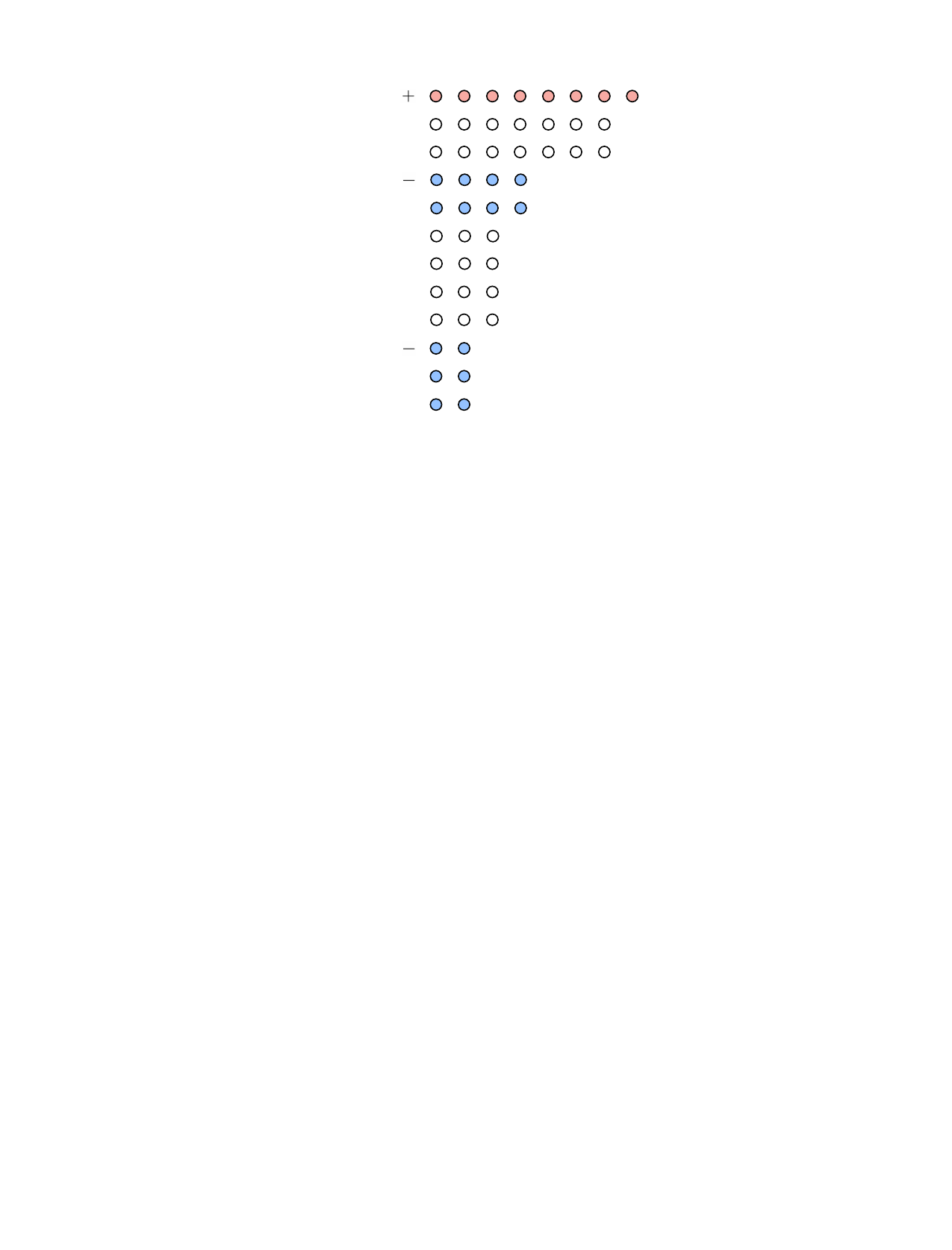}}
\end{figure}
If $q$ is even, then a precise one-to-one combinatorial description of the conjugacy classes is more delicate to see directly from \cite{Wall}, and we refer the reader to \cite{DeFran} for a recent and comprehensive treatment of the even characteristic case. However, here we are only interested in the $\GL_{2m}(q)$-rational canonical form data of elements of $\Sp_{2m}(q)$. In particular, the number of elements of $\Sp_{2m}(q)$ having such fixed rational canonical data depends on $q$ in a way independent of the characteristic. This was proved by Fulman and Guralnick in the following theorem. 
From now on, all polynomials $\phi \in \mathbb{F}_{q}[z]$ will be assumed to be monic and irreducible, and for $a\in \Sp_{2m}(q)$, let $\lambda_{\phi}=\lambda_{\phi}(a)$ be the partition associated to $a$ via the polynomial $\phi$ by means of the $\GL_{2m}(q)$-rational canonical form of $a$.
Define
$$B(q,\lambda_{\phi})\coloneqq
\begin{cases}
(-1)^{|\lambda_{\phi}|}c_{\GL,-q^{\deg \phi/2}}(\lambda_{\phi}) & \text{ if } \phi = \bar{\phi} \text{ and }\phi\neq z\pm1, \\
(c_{\GL,q^{\deg \phi}}(\lambda_{\phi}))^{1/2} & \text{ if } \phi \neq \bar{\phi}.
\end{cases}$$
See also \cite{Fulman_cycle_index} for explicit computations of the rewritings of $B(q,\lambda_{\phi})$.
\begin{lemma}[ \cite{FG04}, Theorem 5.2]
\label{lemma_prop_unip_sp} Let $q$ be either odd or even.
 The proportion of elements of $\Sp_{2m}(q)$ with $\GL_{2m}(q)$-rational canonical form data $(\lambda_{\phi})_{\phi}$
 is 0 unless $\lambda_{\phi}=\lambda_{\bar{\phi}}$ for all  $\phi$, and 
 all odd parts of $\lambda_{z\pm 1}$ occur with even multiplicity. If these conditions are satisfied, then the proportion is
$$\prod_{\phi=z\pm1}\frac{1}{q^{\frac{1}{2}\sum_i((\lambda_{\phi})'_i)^2+\frac{1}{2}o(\lambda_{\phi})}\prod_i(1/q^2)_{\lfloor\frac{m_i(\lambda_{\phi})}{2}\rfloor}}\prod_{\phi \neq {z\pm1}}\frac{1}{B(q,\lambda_{\phi})}.$$
\end{lemma}
We define the \textit{cycle index} $Z_{\Sp}$ of the symplectic groups by
$$Z_{\Sp}\coloneqq 1 + \sum_{m=1}^{\infty}\frac{y^{2m}}{|\Sp_{2m}(q)|}\sum_{a \in \Sp_{2m}(q)}\prod_{\phi}x_{\phi,\lambda_{\phi}(a)}.$$
In \cite[Section~4.2]{Fulman_cycle_index}, Fulman defines and factorises the cycle index for symplectic groups in odd characteristic by assigning a variable $x_{\phi,\lambda^{\pm}}$ to each symplectic signed partition $\lambda^{\pm}$ when $\phi = z \pm 1$. 
Since in the present paper we are only interested in the underlying unsigned partition of $\lambda^{\pm}$, we adopt the above formulation. 
This choice also allows for a unified treatment including the even characteristic case.

From Lemma \ref{lemma_prop_unip_sp} (see also \cite[Theorem 12]{Fulman_cycle_index}), $Z_{\Sp}$ admits the following factorisation:

\begin{align*}
Z_{\Sp}=\prod_{\phi=z\pm1}&\sum_{\substack{|\lambda|=2m\\ i \text{ odd }\Rightarrow m_i(\lambda) \text{ even}}}x_{\phi,\lambda}\frac{y^{|\lambda|}}{q^{\frac{1}{2}\sum_i(\lambda'_i)^2 + \frac{1}{2}o(\lambda)}\prod_i(1/q^2)_{\lfloor\frac{m_i(\lambda)}{2}\rfloor}}\\ \cdot \prod_{\substack{\phi=\bar{\phi}\\\phi \neq z\pm1}}&\sum_{\lambda}x_{\phi,\lambda}\frac{(-y^{\deg \phi})^{|\lambda|}}{c_{\GL,-q^{\deg \phi /2}}(\lambda)}
\cdot
\prod_{\substack{\{\phi,\bar{\phi}\}\\\phi \neq \bar{\phi}}}\sum_{\lambda}x_{\phi,\lambda}x_{\bar{\phi},\lambda}
\frac{y^{2|\lambda|\deg \phi}}{c_{\GL,q^{\deg \phi}}(\lambda)}.
\end{align*}

In this section, for each $m\ge1$, we denote by $d_m(q)$ and $u_m(q)$ the values $\delta(\ASp_{2m}(q))$ and $\delta_p(\ASp_{2m}(q))$, respectively; and we write $d'_m(q)$ and $u'_m(q)$ for $\delta'(\Sp_{2m}(q))$ and $\delta_p'(\Sp_{2m}(q))$. Moreover, we set $d'_0(q)=u'_0(q)\coloneqq 1$.
Explicitly, in view of Equations \eqref{d'(X_m(q))} and \eqref{d'_p(X_m(q))}, and Lemma \ref{steinberg}\eqref{p_sp}, we have
\begin{align}
d'_m(q)&=\frac{1}{|\Sp_{2m}(q)|}\sum_{a \in \Sp_{2m}(q)}q^{-\dim \ker(a-1)}, \\
d_m(q)&=1-d'_m(q), \\
 u'_m(q)&=\frac{1}{|\Sp_{2m}(q)|}\sum_{\substack{a \in \Sp_{2m}(q) \\a \text{ unipotent}}}q^{-\dim \ker(a-1)},\\
 u_m(q)&=\frac{1}{q^m(1/q^2)_m}-u'_m(q).
\end{align}

As in the unitary case, we define the generating functions
\begin{align}
D_{\Sp} \coloneqq \sum_{m=1}^{\infty}d_m(q)y^{2m}, &&
U_{\Sp} \coloneqq \sum_{m=1}^{\infty}u_m(q)y^{2m}, \\
\label{D'_Sp}
D'_{\Sp}\coloneqq\sum_{m=0}^{\infty}d'_m(q)y^{2m}, && U'_{\Sp} \coloneqq\sum_{m=0}^{\infty}u'_m(q)y^{2m}.
\end{align}
\subsection{Derangements of $p$-power order: proof of Theorem \ref{p_main}\eqref{p_Symplectic}}

Let $a \in \Sp_{2m}(q)$ be unipotent. According to Lemma \ref{lemma_prop_unip_sp}, the $\GL_{2m}(q)$-rational canonical form of $a$ is determined by a partition $\lambda$ of $2m$ in which all odd parts have even multiplicity.

\begin{lemma}
\label{lemma_p_sympl}
$$u_m(q)=\sum_{\substack{|\lambda|=2m\\ i \text{ odd }\Rightarrow m_i \text{ even}}}\frac{1-q^{-\lambda'_1}}{q^{\frac{1}{2}\sum_i(\lambda'_i)^2+\frac{1}{2}o(\lambda)}\prod_i(1/q^2)_{\lfloor\frac{m_i}{2}\rfloor}}.$$
\begin{proof} Let $a\in \Sp_{2m}(q)$ be unipotent, and let $\lambda_{z-1}(a)$ be its $\GL_{2m}(q)$-rational canonical form type, that is $\lambda_{z-1}(a)$ is a partition of $2m$ in which all odd parts occur with even multiplicity.
Recalling that $\dim \ker(a-1)=\lambda_{z-1}(a)'_1$, the first equality in Eq. \eqref{prop_unip_der} yields
$$u_m(q)=\frac{1}{|\Sp_{2m}(q)|}\sum_{\substack{a \in \Sp_{2m}(q)\\a \text{ unipotent}}}\left(1-q^{-\dim \ker(a-1)}\right)=\frac{1}{|\Sp_{2m}(q)|}\sum_{\substack{a \in \Sp_{2m}(q)\\a \text{ unipotent}}}\left(1-q^{-\lambda_{z-1}(a)'_1}\right) .$$
Now, the conclusion immediately follows from the definition of the cycle index $Z_{\Sp}$. 
\end{proof}
\end{lemma}

The following identity, proven in \cite{FulmanStanton25}  by Fulman and Stanton, proves Theorem \ref{p_main}\eqref{p_Symplectic}.
\begin{thm*}[Restatement of Theorem \ref{conj_identities}\eqref{conj_intro_sympl}]
$$u_m(q)=\sum_{\substack{|\lambda|=2m\\ i \text{ odd }\Rightarrow m_i \text{ even}}}\frac{1-q^{-\lambda'_1}}{q^{\frac{1}{2}\sum_i(\lambda'_i)^2+\frac{1}{2}o(\lambda)}\prod_i(1/q^2)_{\lfloor\frac{m_i}{2}\rfloor}}=
\frac{1}{q^m(q+1)}\sum_{i=1}^{m}\frac{(-1)^{i-1}(q^{2i+1}+1)}{q^{i(i+1)}(1/q^2)_{m-i}}.
$$
\end{thm*}

\subsection{Derangements: proof of Theorem \ref{main}\eqref{Symplectic}}
Define
\begin{equation}
\label{T_Sp}
T_{\Sp}\coloneqq\sum_{m=0}^{\infty}\frac{1}{q^m(1/q^2)_m}y^{2m}.
\end{equation}
Recall the definitions of $D'_{\Sp}$ and $U'_{\Sp}$  in Eq. \eqref{D'_Sp}.
Analogously to Lemma \ref{fact_D'_u}, $D'_{\Sp}$ factorises as follows.
\begin{lemma}
\label{fact_D'_sp}
We have
$$D'_{\Sp}=T^{-1}_{\Sp}(1-y^2)^{-1}U'_{\Sp}.$$
\begin{proof}
Proceed as in the proof of Lemma \ref{fact_D'_u}, noting that, if we set all variables of $Z_{\Sp}$ equal to $1$, we obtain $\frac{1}{1-y^2}$, and if we set all variables $x_{z-1,\lambda}$ equal to 1 in the factor 
$$\sum_{\substack{|\lambda|=2m\\ i \text{ odd }\Rightarrow m_i(\lambda) \text{ even}}}x_{z-1,\lambda}\frac{y^{|\lambda|}}{q^{\frac{1}{2}\sum_i(\lambda'_i)^2+\frac{1}{2}o(\lambda)}\prod_i(1/q^2)_{\lfloor\frac{m_i(\lambda)}{2}\rfloor}}$$

of $Z_{\Sp}$, then the coefficient of $y^{2m}$ in this sum equals the proportion of unipotent elements of $\Sp_{2m}(q)$. By Lemma \ref{steinberg}, this proportion is $\frac{1}{q^m(1/q^2)_m}$.
Therefore,
\begin{equation*}
\sum_{\substack{|\lambda|=2m\\ i \text{ odd }\Rightarrow m_i(\lambda) \text{ even}}}\frac{y^{|\lambda|}}{q^{\frac{1}{2}\sum_i(\lambda'_i)^2+\frac{1}{2}o(\lambda)}\prod_i(1/q^2)_{\lfloor\frac{m_i(\lambda)}{2}\rfloor}}=\sum_{m=0}^{\infty}\frac{1}{q^m(1/q^2)_m}y^{2m}=T_{\Sp}.
\qedhere
\end{equation*} 
\end{proof}
\end{lemma}

We can now prove Theorem \ref{main}\eqref{Symplectic}.
\begin{proof}[Proof of Theorem $\ref{main}\eqref{Symplectic}$]
Applying Theorem \ref{conj_identities}\eqref{conj_intro_sympl} we get:
\begin{align*}
u'_m(q)& =\frac{1}{q^m(1/q^2)_m}-u_m(q) \\
& =\frac{1}{q^m(1/q^2)_m}-\frac{1}{q^m(q+1)}\sum_{i=1}^{m}\frac{(-1)^{i-1}(q^{2i+1}+1)}{q^{i(i+1)}(1/q^2)_{m-i}}\\
& = \frac{1}{q^m(q+1)}\sum_{i=0}^{m}\frac{(-1)^{i}(q^{2i+1}+1)}{q^{i(i+1)}(1/q^2)_{m-i}}.
\end{align*}
Now, note that the following factorisation for $U'_{\Sp}$ holds:
\begin{align}
    \nonumber U'_{\Sp}=\sum_{m=0}^{\infty}u'_m(q)y^{2m}&=\sum_{m=0}^{\infty}\frac{1}{q+1}\sum_{i=0}^{m}
    \frac{(-1)^{i}(q^{2i+1}+1)}{q^{i(i+2)}q^{m-i}(1/q^2)_{m-i}}
    y^{2m}\\
    \nonumber &=\sum_{m=0}^{\infty}\frac{1}{q+1}\sum_{i=0}^{m}\frac{(-1)^i(q^{2i+1}+1)}{q^{i(i+2)}}y^{2i}\cdot\frac{1}{q^{m-i}(1/q^2)_{m-i}}y^{2(m-i)}\\
    \label{fatt_sp}
&=\overline{D}_{\Sp}\cdot T_{ \Sp},
\end{align}
where we have defined
$$\overline{D}_{\Sp}:=\frac{1}{q+1}\sum_{i=0}^{\infty}\frac{(-1)^i(q^{2i+1}+1)}{q^{i(i+2)}}y^{2i}.$$
By Eq. \eqref{fatt_sp} and Lemma \ref{fact_D'_sp}, we have 
\begin{align*}
D'_{\Sp}&=T^{-1}_{\Sp}\cdot U'_{\Sp} \cdot (1-y^2)^{-1}= \overline{D}_{\Sp} \cdot (1-y^2)^{-1} \\
&=\frac{1}{q+1}\sum_{i=0}^{\infty}\frac{(-1)^i(q^{2i+1}+1)}{q^{i(i+2)}}y^{2i}\sum_{j=0}^{\infty}y^{2j}=\sum_{m=0}^{\infty}\frac{1}{q+1}\sum_{i=0}^{m}\frac{(-1)^i(q^{2i+1}+1)}{q^{i(i+2)}}y^{2m}.
\end{align*}
Hence,
$$d'_m(q)=\frac{1}{q+1}\sum_{i=0}^{m}\frac{(-1)^i(q^{2i+1}+1)}{q^{i(i+2)}}$$
and
$$d_m(q)=1-d'_m(q)=\frac{1}{q+1}\sum_{i=1}^{m}\frac{(-1)^{i-1}(q^{2i+1}+1)}{q^{i(i+2)}}=\frac{1}{q+1}\left(1-\frac{1}{(-q)^{m(m+2)}}\right),$$
where the last equality follows easily by induction.
\end{proof}

\section{Affine orthogonal groups}
\label{sec_orthogonal}
We start by recalling some terminology concerning bilinear forms. A non-degenerate bilinear form $N$ is called \textit{null} if the vector space $V$ on which it is defined can be decomposed as the direct sum of two totally isotropic subspaces, that is, as the sum of two subspaces on which the form vanishes for all pairs of vectors.

Two non-degenerate bilinear forms $Q$ and $Q'$ are said to be \textit{equivalent} if $Q'$ is isomorphic to the direct sum of $Q$ and a null form $N$. The \textit{Witt type} of $Q$ is the equivalence class of $Q$ under this equivalence relation. Over $\mathbb{F}_q$, there exist exactly four Witt types, denoted by Wall \cite{Wall} as $\mathbf{0},\mathbf{1},\delta,\omega$, which correspond to the forms $x^2,\delta x^2,x^2-\delta y^2$ respectively, where $\delta$ is a fixed non-square element in $\mathbb{F}_q$. The set of Witt types is closed under a natural addition: the sum of two Witt types, represented by forms $Q_1$ and $Q_2$ on spaces $V_1$ and $V_2$ is the equivalence class of $Q_1\oplus Q_2$ on $V_1\oplus V_2$.

The four orthogonal groups $\Or^{+}_{2n+1}(q),\Or^{-}_{2n+1}(q),\Or^{+}_{2n}(q)$ and $\Or^{-}_{2n}(q)$ are the subgroups of $\GL_{m}(q)$ (where $m$ is, accordingly, $2n+1$ or $2n$) preserving a non-degenerate quadratic form of Witt type $\mathbf{1},\delta,\mathbf{0}$ and $\omega$, respectively. We remark that the groups $\Or^{+}_{2n+1}(q)$ and $\Or^{-}_{2n+1}(q)$ are isomorphic,
and we usually denote them simply by $\Or_{2n+1}(q)$. However, Wall's parametrisation of conjugacy classes in orthogonal groups distinguishes between them, and therefore,
throughout this section, it will sometimes be convenient to maintain this distinction.

The description of the conjugacy classes of the finite orthogonal groups differs according to the parity of the characteristic of the defining field and, for clarity,
we choose to treat the two cases in separate sections.
Before addressing the odd characteristic case, we fix the common notation.
For $m\ge1$ let
\begin{multicols}{2}
\begin{itemize}[label={}]
    \item $d_m(q) \coloneqq \delta(\AO^+_m(q))+\delta(\AO^-_m(q))$,
    \item $\bar{d}_m(q)\coloneqq \delta(\AO^+_m(q))-\delta(\AO^-_m(q))$,
    \item $d'_m(q) \coloneqq \delta'(\Or^+_m(q))+\delta'(\Or^-_m(q))$,
    \item $\bar{d}'_m(q) \coloneqq \delta'(\Or^+_m(q))-\delta'(\Or^-_m(q))$,
      \item $u_m(q) \coloneqq \delta_p(\AO^+_m(q))+\delta_p(\AO^-_m(q))$,
    \item $\bar{u}_m(q)\coloneqq \delta_p(\AO^+_m(q))-\delta_p(\AO^-_m(q))$,
    \item $u'_m(q) \coloneqq \delta'_p(\Or^+_m(q))+\delta'_p(\Or^-_m(q))$,
    \item $\bar{u}'_m(q) \coloneqq \delta'_p(\Or^+_m(q))-\delta'_p(\Or^-_m(q))$.
\end{itemize}
\end{multicols}
Moreover, we set $d'_0(q)=u'_0(q)=\bar{d}'_0(q)=\bar{u}'_0(q)\coloneqq 1$. 
Note that the quantities related to the difference of the orthogonal groups are non-trivial only when the dimension $m$ is even.

In view of Equations \eqref{d'(X_m(q))} and \eqref{d'_p(X_m(q))}, we have, for every $m \ge 1$:
\begin{align}
\nonumber
d'_m(q)&=\frac{1}{|\Or^{+}_{m}(q)|}\sum_{a \in \Or^{+}_{m}(q)}q^{-\dim \ker(a-1)}+\frac{1}{|\Or^{-}_{m}(q)|}\sum_{a \in \Or^{-}_{m}(q)}q^{-\dim \ker(a-1)}, \\
\nonumber
d_m(q)&=2-d'_m(q), \\
\nonumber
 u'_m(q)&=\frac{1}{|\Or^{+}_{m}(q)|}\sum_{\substack{a \in \Or^{+}_{m}(q) \\ a \text{ unipotent}}}q^{-\dim \ker(a-1)}+\frac{1}{|\Or^{-}_{m}(q)|}\sum_{\substack{a \in \Or^{-}_{m}(q) \\ a \text{ unipotent}}}q^{-\dim \ker(a-1)}, \\
 \label{sum_unip_qeven}
 u_m(q)&=
 \begin{cases}\frac{1}{q^{\lfloor m/2\rfloor}(1/q^2)_{\lfloor m/2 \rfloor}}-u'_m(q), & \text{ if }q \text{ odd}, \\
 \frac{1}{q^m(1/q^2)_m}+\frac{1}{q^{m-1}(1/q^2)_{m-1}}-u'_m(q), & \text{ if } q \text{ even};
 \end{cases}
\end{align}
where the equalities for $u_m(q)$ follow from Lemma \ref{steinberg}\eqref{p_o_2m1}-\eqref{p_o2m}, as, for $q$ odd,
\begin{align*}
\Delta_u(\Or^{+}_{2m}(q))+\Delta_u(\Or^{-}_{2m}(q))=& \frac{q^{m^2-m}}{2\prod_{i=1}^{m-1}(q^{2i}-1)}\left(\frac{1}{q^m-1}+\frac{1}{q^m +1}\right)\\
=&\frac{q^{m^2}}{\prod_{i=1}^{m}(q^{2i}-1)}=\frac{q^{m^2}}{q^{m^2+m}(1/q^2)_m}=\frac{1}{q^m(1/q^2)_m},
\end{align*}
and
\begin{align*}
\Delta_u(\Or^{+}_{2m+1}(q))+\Delta_u(\Or^{-}_{2m+1}(q))=2\Delta_u(\Or^{+}_{2m+1}(q))=\frac{1}{q^m(1/q^2)_m};
\end{align*}
while, for $q$ even,
\begin{align*}
\Delta_u(\Or^{+}_{2m}&(q))+\Delta_u(\Or^{-}_{2m}(q))=\frac{q^{(m-1)^2}}{2\prod_{i=1}^{m-1}(q^{2i}-1)}\left( \frac{q^m+q^{m-1}-1}{q^m-1}+\frac{q^m+q^{m-1}+1}{q^m+1}\right)\\
=&\frac{q^{(m-1)^2}}{2\prod_{i=1}^{m}(q^{2i}-1)}\left(2q^{2m}+2q^{2m-1}-2\right)=\frac{q^{m^2-2m+1}}{q^{m^2+m}(1/q^2)_m}\left(q^{2m}+q^{2m-1}-1\right) \\
=&\frac{1+q-q^{1-2m}}{q^{m}(1/q^2)_m}=\frac{1}{q^m(1/q^2)_m}+\frac{q(1-q^{-2m)}}{q^m(1/q^2)_m}
=\frac{1}{q^m(1/q^2)_m}+\frac{1}{q^{m-1}(1/q^2)_{m-1}}.
\end{align*}
Moreover, we have
\begin{align}
\nonumber
\bar{d}'_{2m}(q)&=\frac{1}{|\Or^{+}_{2m}(q)|}\sum_{a \in \Or^{+}_{2m}(q)}q^{-\dim \ker(a-1)}-\frac{1}{|\Or^{-}_{2m}(q)|}\sum_{a \in \Or^{-}_{2m}(q)}q^{-\dim \ker(a-1)}, \\
\nonumber
\bar{d}_{2m}(q)&=-\bar{d}'_{2m}(q), \\
\nonumber
 \bar{u}'_{2m}(q)&=\frac{1}{|\Or^{+}_{2m}(q)|}\sum_{\substack{a \in \Or^{+}_{2m}(q) \\ a \text{ unipotent}}}q^{-\dim \ker(a-1)}-\frac{1}{|\Or^{-}_{2m}(q)|}\sum_{\substack{a \in \Or^{-}_{2m}(q) \\ a \text{ unipotent}}}q^{-\dim \ker(a-1)},\\
 \label{difference_u_ort}
 \bar{u}_{2m}(q)&=\frac{1}{q^{2m}(1/q^2)_{m}}-\bar{u}'_{2m}(q),
\end{align}
where the equality for $\bar{u}_{2m}(q)$ follows from Lemma \ref{steinberg}\eqref{p_o2m}, as, for $q$ odd,
\begin{align*}
\Delta_u(\Or^{+}_{2m}(q))-\Delta_u(\Or^{-}_{2m}(q))=& \frac{q^{m^2-m}}{2\prod_{i=1}^{m-1}(q^{2i}-1)}\left(\frac{1}{q^m-1}-\frac{1}{q^m +1}\right)\\
=&\frac{q^{m^2-m}}{\prod_{i=1}^{m}(q^{2i}-1)}=\frac{q^{m^2-m}}{q^{m^2+m}(1/q^2)_m}=\frac{1}{q^{2m}(1/q^2)_m}
\end{align*}
and, for $q$ even,
\begin{align*}
\Delta_u(\Or^{+}_{2m}&(q))+\Delta_u(\Or^{-}_{2m}(q))=\frac{q^{(m-1)^2}}{2\prod_{i=1}^{m-1}(q^{2i}-1)}\left( \frac{q^m+q^{m-1}-1}{q^m-1}-\frac{q^m+q^{m-1}+1}{q^m+1}\right)\\
=&\frac{q^{(m-1)^2}2q^{m-1}}{2\prod_{i=1}^{m}(q^{2i}-1)}=\frac{q^{m^2-m}}{q^{m^2+m}(1/q^2)_m}=\frac{1}{q^{2m}(1/q^2)_m}.
\end{align*}
Finally, we define the generating functions
\begin{align}
D_{\Or} \coloneqq \sum_{m=1}^{\infty}d_m(q)y^m, &&
U_{\Or} \coloneqq \sum_{m=1}^{\infty}u_m(q)y^m, &&
\bar{D}_{\Or} \coloneqq \sum_{m=1}^{\infty}\bar{d}_{2m}(q)y^{2m}, &&
\bar{U}_{\Or} \coloneqq \sum_{m=1}^{\infty}\bar{u}_{2m}(q)y^{2m},
\\
\label{D'_Or}
D'_{\Or}\coloneqq\sum_{m=0}^{\infty}d'_m(q)y^m, && U'_{\Or} \coloneqq\sum_{m=0}^{\infty}u'_m(q)y^m, &&\bar{D}'_{\Or} \coloneqq \sum_{m=1}^{\infty}\bar{d}'_{2m}(q)y^{2m}, &&
\bar{U}'_{\Or} \coloneqq \sum_{m=1}^{\infty}\bar{u}'_{2m}(q)y^{2m}.
\end{align}
In the sequel, we will often refer to \textit{the sum} or \textit{the difference of the orthogonal groups}, meaning the sum or the difference of the corresponding quantities for $\Or^{+}_{m}(q)$ and $\Or^{-}_{m}(q)$, as in all the definitions above.
\section{Affine orthogonal groups in odd characteristic}
In this section, $q$ is always assumed to be odd.
We follow, as before, the exposition in \cite[Sec. 4.3]{Fulman_cycle_index}. Let $\phi$ be a polynomial over $\mathbb{F}_q$ with non-zero constant term. 
We consider the following combinatorial data: for each monic, non-constant, irreducible polynomial $\phi \neq z\pm 1$ assign a partition $\lambda_{\phi}$ of some non-negative integer $|\lambda_{\phi}|$. For $\phi$ equal to $z-1$ or $z+1$ associate an \textit{orthogonal signed partition} $\lambda^{\pm}_{\phi}$, meaning a partition of a non-negative integer $|\lambda^{\pm}_{\phi}|$ such that all even parts occur with even multiplicity, and each odd $i>0$ is assigned a sign (see Figure \ref{fig:orthogonal_signed_partition}). For such $\phi=z\pm1$ and odd $i>0$, let $\Theta_i(\lambda^{\pm}_{\phi})$ denote the Witt type of the orthogonal group on a vector space of dimension $m_i(\lambda^{\pm}_{\phi})$, with the sign determined by the choice assigned to the part $i$. Recall, from the discussion in the symplectic case, the definition of $\bar{\phi}$ in Eq. \eqref{bar_phi}. 

\begin{figure}
\floatbox[{\capbeside\thisfloatsetup{capbesideposition={right,bottom},capbesidewidth= 12cm}}]{figure}[\FBwidth]
{\caption{\textbf{Example of an orthogonal signed partition:} here the $+$ corresponds to the part of size 3 and the $-$ corresponds to the parts of size 5 and 1.}\label{fig:orthogonal_signed_partition}}
{\includegraphics[]{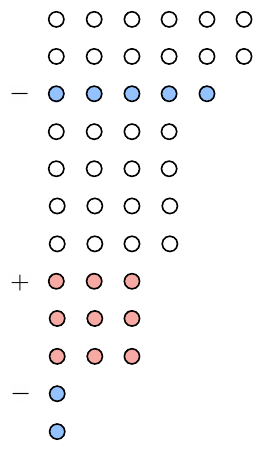}}
\end{figure}

\begin{thm}[\cite{Fulman_cycle_index}, Theorem 13]
\label{Wall_orthogonal}
The data $\lambda^{\pm}_{z-1},\lambda^{\pm}_{z+1},\lambda_{\phi}$ represent a conjugacy class of some orthogonal group $\Or^{\epsilon}_{m}(q)$, $\epsilon \in \{+,-,\circ\}$, if and only if
\begin{enumerate}[label=(\roman*)]
    \item $|\lambda_z|=0,$
    \item $\lambda_{\phi}=\lambda_{\bar{\phi}}$,
    \item $\sum_{\phi=z\pm1}|\lambda_{\phi}^{\pm}|+\sum_{\phi \neq z \pm1}|\lambda_{\phi}|\deg (\phi)=m$.
\end{enumerate}

This partition data represents the conjugacy class of exactly one orthogonal group $\Or^{\epsilon}_m(q)$, where the sign $\epsilon$ is uniquely determined by the condition that the group arises as the stabiliser of a form of Witt type
$$\sum_{\phi=z\pm 1}\sum_{i\text{ odd}}\Theta_i(\lambda^{\pm}_{\phi})+\sum_{\phi \neq z\pm 1}\sum_{i \ge1}i m_i(\lambda_{\phi})\omega.$$
\end{thm}
If $a \in \Or^{\pm}_{m}(q)$, we denote by $\lambda_{\phi}(a)$ (resp. $\lambda^{\pm}_{\phi}(a)$, if $\phi=z\pm 1$) the partition (resp. orthogonal signed partition) associated with $a$ corresponding to the irreducible polynomial~$\phi$.
For every $m \ge 1$, and $\epsilon \in \{\pm\}$ define
$$Z_{\Or^{\epsilon}_m(q)} \coloneqq 
\frac{1}{|\Or^{\epsilon}_{m}(q)|}\sum_{a \in \Or^{\epsilon}_{m}(q)}\prod_{\phi=z\pm1}x_{\phi,\lambda^{\pm}_{\phi}(a)}\prod_{\phi\neq z,z\pm1}x_{\phi,\lambda_{\phi}(a)},$$
Then, the \textit{cycle index for the sum of the orthogonal groups} is defined as
\begin{align*}
Z_{\Or} \coloneqq 1 + \sum_{m=1}^{\infty} Z_{\Or^{+}_m(q)} y^{m} + Z_{\Or^{-}_m(q)} y^{m}.
\end{align*}
Define
\begin{align*}
A_i(q,\lambda^{\pm})&\coloneqq
\begin{cases}
|\Or_{m_i(\lambda^{\pm})}(q)| & \text{if } i \equiv 1 \pmod 2, \\
    q^{-\frac{1}{2}m_i(\lambda^{\pm})}|\Sp_{m_i(\lambda^{\pm})}(q)| & \text{if } i\equiv 0 \pmod 2.
\end{cases}\\
\end{align*}
and
\begin{equation}
\label{c_O}
c_{\Or,q}(\lambda^{\pm})\coloneqq q^{\sum_i((\lambda^{\pm})'_i)^2-\sum_im_i(\lambda^{\pm})^2}\prod_iA_i(q,\lambda^{\pm}).
\end{equation}
From \cite[Theorem 14]{Fulman_cycle_index}, $Z_{\Or}$ admits the following factorisation.
\begin{align}
\label{factorisation_ci_sum_ort}
Z_{\Or}=\prod_{\phi=z\pm1}&\left(\sum_{\lambda^{\pm}}x_{\phi,\lambda^{\pm}}\frac{y^{|\lambda^{\pm}|}}{c_{\Or,q}(\lambda^{\pm})}\right)\prod_{\substack{\phi=\bar{\phi}\\\phi \neq z\pm1}}\sum_{\lambda}x_{\phi,\lambda}\frac{(-y^{\deg \phi})^{|\lambda|}}{c_{\GL,-q^{\deg \phi/2}}(\lambda)} \nonumber\\
\cdot&
\prod_{\substack{\{\phi,\bar{\phi}\}\\\phi \neq \bar{\phi}}}\sum_{\lambda}x_{\phi,\lambda}x_{\bar{\phi},\lambda}
\frac{y^{2|\lambda|\deg \phi}}{c_{\GL,q^{\deg \phi}}(\lambda)}.
\end{align}
Note that, in order to derive formulas for $\delta({\AO^{+}_{m}(q)})$ and  $\delta({\AO^{-}_{m}(q)})$ when $m$ is even, we also need a cycle index for the \textit{difference} of the orthogonal groups. To this end, for the odd-characteristic case, we follow Britnell's work \cite{Britnell_difference_orth}, which treats separately the cases $q \equiv 1 \pmod 4$ and $q \equiv 3\pmod 4$.

For an orthogonal signed partition $\lambda^{\pm}$, and for odd $s$, let $\sign(s)$ denote the sign corresponding to the parts of $\lambda^{\pm}$ of size $s$. Define, for  $\phi \in {z\pm 1}$,
\begin{equation*}
 \tau_{\phi,s}\coloneqq
\begin{cases}
    \sign(s) & \text{ if }q^{m_s} \equiv 1 \pmod 4, \\
    i\sign(s)&  \text{ if }q^{m_s} \equiv 3 \pmod 4,
\end{cases}
\end{equation*}
 where $i$ denotes the imaginary unity, and 
 \begin{equation}
 \label{tau_phi}
\tau_{\phi}(\lambda^{\pm}) \coloneqq \prod_{s \text{ odd}}\tau_{\phi,s}.
 \end{equation}

Now, define $Z_{\Or}^{[1]}$ to be the power series obtained from $Z_{\Or}$ under the substitutions
\begin{align}
x_{\phi, \lambda^{\pm}}\mapsto \tau_{\phi}(\lambda^{\pm})x_{\phi, \lambda^{\pm}},&& x_{\phi, \lambda} \mapsto (-1)^{|\lambda|}x_{\phi,\lambda}.
\end{align}
In \cite[p. 580]{Britnell_difference_orth} it is shown that, if $q \equiv 1 \pmod 4$, then
\begin{equation}
Z_{\Or}^{[1]}=1+\sum_{m=1}^{\infty} Z_{\Or^{+}_{m}(q)}y^m-Z_{\Or^{-}_{m}(q)}y^m
\end{equation}
is the cycle index for the \text{difference} of the orthogonal groups, and if $q \equiv 3 \pmod 4$, then
\begin{equation}
Z_{\Or}^{[1]}=1+\sum_{n=1}^{\infty}i(Z_{\Or^{+}_{2n-1}(q)}y^{2n-1}-Z_{\Or^{-}_{2n-1}(q)}y^{2n-1})+Z_{\Or^{+}_{2n}(q)}y^{2n}-Z_{\Or^{-}_{2n}(q)}y^{2n}.
\end{equation}
Finally, observe that, using Eq. \eqref{factorisation_ci_sum_ort}, $Z^{[1]}_{\Or}$ factorises as follows:
\begin{align}
\label{factorisation_ci_difference_ort}
Z^{[1]}_{\Or}=\prod_{\phi=z\pm1}&\left(\sum_{\lambda^{\pm}}\tau_{\phi,\lambda^{\pm}}
x_{\phi,\lambda^{\pm}}\frac{y^{|\lambda^{\pm}|}}{c_{\Or,q}(\lambda^{\pm})}\right)\prod_{\substack{\phi=\bar{\phi}\\\phi \neq z\pm1}}\sum_{\lambda}x_{\phi,\lambda}\frac{y^{|\lambda|\deg \phi}}{c_{\GL,-q^{\deg \phi/2}}(\lambda)} \nonumber \nonumber \\
\cdot&
\prod_{\substack{\{\phi,\bar{\phi}\}\\\phi \neq \bar{\phi}}}\sum_{\lambda}x_{\phi,\lambda}x_{\bar{\phi},\lambda}
\frac{y^{2|\lambda|\deg \phi}}{c_{\GL,q^{\deg \phi}}(\lambda)}.
\end{align}

\subsection{Derangements of $p$-power order: proof of Theorem \ref{p_main}\eqref{p_Orthogonal_2m+1}-\eqref{p_Orthogonal_2modd}}
Let $a \in \Or^{\epsilon}_{m}(q)$ be unipotent. According to Theorem \ref{Wall_orthogonal}, the conjugacy class of $a$ corresponds to an orthogonal signed  partition $\lambda^{\pm}_{z-1}(a)$ of size $m$. Conversely, if $\lambda^{\pm}$ is an orthogonal-signed partition of size $m$, then $\lambda^{\pm}$ determines a unique conjugacy class of a unipotent element $a \in \Or^{\epsilon}_{m}(q)$, where the sign $\epsilon$ is uniquely specified by the condition that the group arises as the stabiliser of a form of Witt type
$\sum_{i\text{ odd}}\Theta_i(\lambda^{\pm}).$
We say that the unipotent element $a \in \Or^{\epsilon}_{m}(q)$ has rational canonical form of \textit{type} $\lambda$ if the underlying (unsigned) partition of $\lambda^{\pm}_{z-1}(a)$, which describes the $\GL_{m}(q)$-rational canonical form of $a$,
is equal to $\lambda$. In this case, we write $u(\lambda^{\pm}_{z-1}(a))=\lambda$.

Before deriving, as in the previous cases, the expressions of $u_m(q)$ and 
$\bar u_m(q)$ from the cycle indices, we record the analogue of Lemma~\ref{lemma_prop_unip_sp} for the sum and the difference of the orthogonal groups, which will be useful in deriving these expressions.

\begin{lemma}
\label{unip_orth_lambda}
Let $\lambda$ be a partition of $m$ where all even parts have even multiplicity, and let
$\Delta_{u,\lambda}(\Or^{\epsilon}_m(q))$ be the proportion of unipotent elements in $\Or^{\epsilon}_m(q)$ having rational canonical form of type $\lambda$. Then
\begin{enumerate}[label=\roman*.]
    \item 
$\displaystyle \Delta_{u,\lambda}(\Or^{+}_m(q))+\Delta_{u,\lambda}(\Or^{-}_m(q))=\frac{1}{q^{\frac{1}{2}\sum_i (\lambda'_i)^2-\frac{1}{2}o(\lambda)}\prod_i(1/q^2)_{\lfloor\frac{m_i(\lambda)}{2}\rfloor}};$
\item 
$\displaystyle
\Delta_{u,\lambda}(\Or^{+}_m(q))-\Delta_{u,\lambda}(\Or^{-}_m(q))=\begin{cases}0, &\text{if }\exists i \text{ such that }m_i(\lambda) \text{ is odd;} \\
\frac{1}{q^{\frac{1}{2}\sum_i (\lambda'_i)^2}\prod_i(1/q^2)_{\frac{m_i(\lambda)}{2}}}, &\text{if all }m_i(\lambda) \text{ are even}.\end{cases}$
\end{enumerate}
\begin{proof}
We start by considering $\Delta_{u,\lambda}(\Or^{+}_m(q))+\Delta_{u,\lambda}(\Or^{-}_m(q))$.
     Recall the definition of $c_{\Or,q}(\lambda^{\pm})$ in Eq.~\eqref{c_O}. 
   It follows from the definition of the cycle index $Z_{\Or}$, that $\Delta_{u,\lambda}(\Or^{+}_m(q))+\Delta_{u,\lambda}(\Or^{-}_m(q))$ is the coefficient of $y^m$ in the factor
$$\sum_{\lambda^{\pm}}x_{z-1,\lambda^{\pm}}\frac{y^{|\lambda|}}{c_{\Or,q}(\lambda^{\pm})}$$ of
   $Z_{\Or}$, when we set the variables $x_{z-1,\lambda^{\pm}}$ equal to 1 if $u(\lambda^{\pm})=\lambda$, and equal to 0 otherwise. 
   
    In what follows, after the first step, we will simply write $m_i$ for $m_i(\lambda)$ or $m_i(\lambda^{\pm})$.
    We have
    \allowdisplaybreaks{
    \begin{align*}
    &\Delta_{u,\lambda}(\Or^{+}_m(q))+\Delta_{u,\lambda}(\Or^{-}_m(q))=\sum_{\substack{\lambda^{\pm}\\u(\lambda^{\pm})=\lambda}}\frac{1}{c_{\Or,q}(\lambda^{\pm})} \\&= \frac{1}{q^{\frac{1}{2}\sum_{i}(\lambda'_i)^2-\frac{1}{2}\sum_i m_i(\lambda)^2}} \sum_{\substack{\lambda^{\pm}\\u(\lambda^{\pm})=\lambda}}\frac{1}{\prod_{i \text{ even}}q^{-\frac{m_i(\lambda)}{2}}|\Sp_{m_i(\lambda)}(q)| \prod_{i \text{ odd}}|\Or_{m_i(\lambda^{\pm})}(q)|}\\
   &= \frac{1}{q^{\frac{1}{2}\sum_{i}(\lambda'_i)^2}} \sum_{\substack{\lambda^{\pm}\\u(\lambda^{\pm})=\lambda}}\frac{1}{\displaystyle\prod_{i \text{ even}}q^{\frac{-m_i^2-2m_i}{4}}\prod_{l=1}^{m_i/2}(q^{2l}-1)\prod_{\substack{i \text{ odd}\\m_i \text{ odd}}}2q^{\frac{-m_i^2-2m_i+1}{4}}\prod_{l=1}^{(m_i-1)/2}(q^{2l}-1)}\\ \cdot & \frac{1}{\displaystyle \prod_{\substack{i \text{ odd}\\m_i \text{ even}}}2q^{\frac{-m_i^2-2m_i}{4}}\prod_{l=1}^{m_i/2-1}(q^{2l}-1)\prod_{\substack{i \text{ odd} \\ m_i \text{ even} \\ \text{sign}(i)=+}}(q^{\frac{m_i}{2}}-1)\prod_{\substack{i \text{ odd} \\ m_i \text{ even} \\ \text{sign}(i)=-}}(q^{\frac{m_i}{2}}+1)}.
\end{align*}}
Let $\bar{o}(\lambda)$ denote the number of distinct odd parts of $\lambda$. Then, if $\lambda$ is a fixed partition in which all even parts have even multiplicity, there are exactly $2^{\bar{o}(\lambda)}$ orthogonal signed partitions $\lambda^{\pm}$ with underlying unsigned partition $\lambda$.  Therefore
{\allowdisplaybreaks
\begin{align*}
\Delta_{u,\lambda}(\Or^{+}_m(q))+\Delta&_{u,\lambda}(\Or^{-}_m(q))
=\frac{1}{q^{\frac{1}{2}\sum_{i}(\lambda'_i)^2}}  \cdot \frac{2^{\bar{o}(\lambda)}\prod_{\substack{i \text{ odd}\\ m_i \text{ even}}}q^{\frac{m_i}{2}}}{\displaystyle\prod_{i \text{ even}}q^{\frac{-m_i^2-2m_i}{4}}\prod_{l=1}^{m_i/2}(q^{2l}-1)}\\ \cdot & \frac{1}{\displaystyle\prod_{\substack{i \text{ odd}\\m_i \text{ odd}}}q^{-\frac{m_i}{2}}2q^{\frac{-m_i^2+1}{4}}\prod_{l=1}^{(m_i-1)/2}(q^{2l}-1)\displaystyle \prod_{\substack{i \text{ odd}\\m_i \text{ even}}}2q^{\frac{-m_i^2-2m_i}{4}}\prod_{l=1}^{m_i/2}(q^{2l}-1)}\\
=\frac{1}{q^{\frac{1}{2}\sum_{i}(\lambda'_i)^2}}  &\cdot \frac{q^{\frac{o(\lambda)}{2}}}{\displaystyle\prod_{i \text{ even}}q^{\frac{-m_i^2-2m_i}{4}}\prod_{l=1}^{m_i/2}(q^{2l}-1)\prod_{\substack{i \text{ odd}\\m_i \text{ odd}}}q^{\frac{-m_i^2+1}{4}}\prod_{l=1}^{(m_i-1)/2}(q^{2l}-1)}\\ \cdot & \frac{1}{\displaystyle \prod_{\substack{i \text{ odd}\\m_i \text{ even}}}q^{\frac{-m_i^2-2m_i}{4}}\prod_{l=1}^{m_i/2}(q^{2l}-1)}\\
=\frac{1}{q^{\frac{1}{2}\sum_{i}(\lambda'_i)^2}}  &\cdot \frac{q^{\frac{o(\lambda)}{2}}}{\displaystyle\prod_{i \text{ even}} (1/q^2)_{\frac{m_i}{2}}\prod_{\substack{i \text{ odd}\\m_i \text{ odd}}}(1/q^2)_{\frac{m_i-1}{2}}\prod_{\substack{i \text{ odd}\\m_i \text{ even}}}(1/q^2)_{\frac{m_i}{2}}},
\end{align*}} 
hence the conclusion.

The formula for $\Delta_{u,\lambda}(\Or^{+}_m(q)-\Delta_{u,\lambda}(\Or^{-}_m(q))$ follows from similar computations, recalling the definition of $\tau_{\phi}(\lambda^{\pm})$ in Eq. \eqref{tau_phi} and observing that if $q \equiv 1 \pmod 4$, then $$\tau_{\phi}(\lambda^{\pm})=\prod_{\substack{i \text{ odd} \\ \text{sign}(i)=-}}(-1),$$

while if $q \equiv 3 \pmod 4$, then $$\tau_{\phi}(\lambda^{\pm})=\prod_{\substack{s \text{ even} \\ \text{sign}(s)=-}}(-1)\prod_{\substack{s \text{ odd} \\ \text{sign}(s)=-}}(-i).$$
The presence of the factor $\tau_{\phi}(\lambda^{\pm})$ forces the result to be non-zero if and only if the partitions $\lambda$ has all parts of even multiplicity: indeed, if there exists any $i$ such that $m_i$ is odd, the contributions cancel out due to the sign choices. With this observation, the result follows similarly to the previous case and does not depend on the congruence class of $q\mod 4$.
    \end{proof}
\end{lemma}

As a corollary, we easily obtain exact formulas for the proportion of unipotent elements having fixed rational canonical form type, in each of the orthogonal groups $\Or^{\epsilon}_{m}(q)$.
\begin{coro}
Let $\lambda$ be a partition of $m$. The proportion $\Delta_{u,\lambda}(\Or^{\epsilon}_m(q))$ of unipotent elements in $\Or^{\epsilon}_{m}(q)$ having rational canonical form of type $\lambda$ is 0, unless all even parts of $\lambda$ occur with even multiplicity. 
Assume now that every even part of $\lambda$ occurs with even multiplicity. 
\begin{itemize}
\item If $m$ is odd, or if $m$ is even and there exists $i$ such that $m_i(\lambda)$ is odd, then
\[
\Delta_{u,\lambda}(\Or_m^{\varepsilon}(q))
=
\frac{1}{
2\,q^{\frac12\sum_i (\lambda_i')^2-\frac12 o(\lambda)}
\displaystyle\prod_i (1/q^2)_{\lfloor m_i(\lambda)/2 \rfloor}
}.
\]

\item If $m$ is even and $m_i(\lambda)$ is even for all $i$, then
\[
\Delta_{u,\lambda}(\Or_m^{\pm}(q))
=
\frac{q^{\frac{o(\lambda)}{2}} \pm 1}{
2\,q^{\frac12\sum_i (\lambda_i')^2}
\displaystyle\prod_i (1/q^2)_{m_i(\lambda)/2}
}.
\]
\end{itemize}
\end{coro}
\begin{proof}
By Wall's combinatorial description of the conjugacy classes in the orthogonal groups $\Or^{\epsilon}_m(q)$, the rational canonical form of a unipotent element $a \in \Or^{\epsilon}_{m}(q)$ uniquely corresponds to an orthogonal-signed partition $\lambda^{\pm}_{z-1}(a)$ of size $m$. In particular, the underlying partition $\lambda=u(\lambda^{\pm}_{z-1}(a))$ has all even parts occurring with even multiplicity. This proves that the proportion is $0$ if this requirement is not satisfied. If all even parts of $\lambda$ have even multiplicity, the formulas immediately follow by Lemma \ref{unip_orth_lambda}.
\end{proof}

\begin{lemma}
\label{u_m(q)_ort}
The following hold:
\begin{enumerate}[label=\roman*.]
    \item 
$\displaystyle u_m(q)=\sum_{|\lambda^{\pm}|=m}\frac{1-q^{-(\lambda^{\pm})'_1}}{c_{\Or,q}(\lambda^{\pm})}=\sum_{\substack{|\lambda|=m\\i \text{ even}\Rightarrow m_i(\lambda) \text{ even}}}\frac{1-q^{-\lambda'_1}}{q^{\frac{1}{2}\sum_i (\lambda'_i)^2-\frac{1}{2}o(\lambda)}\prod_i(1/q^2)_{\lfloor\frac{m_i(\lambda)}{2}\rfloor}};$
\item 
$\displaystyle
\bar{u}_{2m}(q)=\sum_{|\lambda^{\pm}|={2m}}\tau_{z-1}(\lambda^{\pm})\frac{1-q^{-(\lambda^{\pm})'_1}}{c_{\Or,q}(\lambda^{\pm})}=\sum_{\substack{|\lambda|=2m\\ \text{all }m_i(\lambda) \text{ even}}}\frac{1-q^{-\lambda'_1}}{q^{\frac{1}{2}\sum_i (\lambda'_i)^2}\prod_i(1/q^2)_{\frac{m_i(\lambda)}{2}}}.$
\end{enumerate}
\end{lemma}
\begin{proof}
In both statements, the first equality follows immediately from the definition of the cycle index $Z_{\Or}$. The second one follows by applying Lemma \ref{unip_orth_lambda}.
\end{proof}

The following identities for $u_{2m+1}(q)$ and ${u}_{2m}(q)$ where conjectured by the author in the first version of this paper and later proved by Fulman and Stanton in \cite{FulmanStanton25}.
\begin{thm*}[Restatement of Theorem $\ref{conj_identities}\eqref{conj_intro_ort_2m+1}-\eqref{conj_intro_ort_2m}$]
The following hold.
\begin{align*}
u_{2m+1}(q)&=\sum_{\substack{|\lambda|=2m+1\\i \text{ even} \\ \Rightarrow m_i \text{ even}}}\frac{1-q^{-\lambda'_1}}{q^{\frac{1}{2}\sum_i (\lambda'_i)^2-\frac{1}{2}o(\lambda)}\prod_i(1/q^2)_{\lfloor\frac{m_i}{2}\rfloor}}=\frac{1}{q^m(1/q^2)_{m}}+\frac{1}{q^{m+1}}\sum_{i=0}^{m}\frac{(-1)^{i-1}}{q^{i(i+1)}(1/q^2)_{m-i}},\\
u_{2m}(q)&=\sum_{\substack{|\lambda|=2m\\i \text{ even}\Rightarrow m_i \text{ even}}}\frac{1-q^{-\lambda'_1}}{q^{\frac{1}{2}\sum_i (\lambda'_i)^2-\frac{1}{2}o(\lambda)}\prod_i(1/q^2)_{\lfloor\frac{m_i}{2}\rfloor}}
=\frac{1}{q^m}\sum_{i=1}^{m}\frac{(-1)^{i-1}}{q^{i(i-1)}(1/q^2)_{m-i}}.
\end{align*}
\end{thm*}

In the case of the difference of the orthogonal groups, the analogue of Theorem \ref{conj_identities}\eqref{conj_intro_ort_2m} can be deduced directly from the corresponding result of \cite{SpigaAGL} on $\AGL_m(q)$. The author would like to thank Tewodros Amdeberhan for helpful comments, which contributed to the proof of this case. 
\begin{prop}
\label{prop_bar{u}}
$$\bar{u}_{2m}(q)=\sum_{\substack{|\lambda|=2m\\ \text{all }m_i \text{ even}}}\frac{1-q^{-\lambda'_1}}{q^{\frac{1}{2}\sum_i (\lambda'_i)^2}\prod_i(1/q^2)_{\frac{m_i}{2}}}=\frac{1}{q^{2m}}\sum_{i=1}^{m}\frac{(-1)^{i-1}}{q^{i(i-1)}(1/q^2)_{m-i}}.$$
\begin{proof}
By Lemma \ref{u_m(q)_ort},
$$\bar{u}_{2m}(q)=\sum_{\substack{|\lambda|=2m\\ \text{all }m_i \text{ even}}}\frac{1-q^{-\lambda'_1}}{q^{\frac{1}{2}\sum_i (\lambda'_i)^2}\prod_i(1/q^2)_{\frac{m_i}{2}}}.$$
This sum ranges over partitions $\lambda$ of $2m$ with even multiplicities, which correspond to partitions $\mu$ of $m$ via $\lambda=2\mu$. In this correspondence, $\lambda'_i=2\mu'_i$, so the expression becomes
$$\bar{u}_{2m}(q)=\sum_{|\mu|=m}\frac{1-q^{-2\mu'_1}}{q^{2\sum_i (\mu'_i)^2}\prod_i(1/q^2)_{m_i(\mu)}}.$$
By Theorem \ref{thmG}, applied with $x=1/q^2$, we obtain the desired conclusion.
\end{proof}
\end{prop}
\begin{proof}[Proof of Theorem \ref{p_main}\eqref{p_Orthogonal_2m+1}-\eqref{p_Orthogonal_2modd}]
For affine orthogonal groups of odd dimension, we obtain
$$\delta_p(\AO_{2m+1}(q))=\frac{u_{2m+1}(q)}{2}=\frac{1}{2q^m(1/q^2)_{m}}+\frac{1}{2q^{m+1}}\sum_{i=0}^{m}\frac{(-1)^{i-1}}{q^{i(i+1)}(1/q^2)_{m-i}}.
$$
For affine orthogonal groups of even dimension, we have
$$\delta_p(\AO^{\pm}_{2m}(q))=\frac{u_{2m}(q)+\bar{u}_{2m}(q)}{2}=\frac{q^m\pm1}{2q^{2m}}\sum_{i=1}^{m}\frac{(-1)^{i-1}}{q^{i(i-1)}(1/q^2)_{m-i}}.$$
\end{proof}
\subsection{Derangements: proof of Theorem \ref{main}\eqref{Orthogonal_odd}-\eqref{Orthogonal_even} }
Define
\begin{align}
T_{\Or} &\coloneqq \sum_{m=0}^{\infty}\frac{y^m}{q^{\lfloor m/2 \rfloor}(1/q^2)_{\lfloor m/2 \rfloor}}, \\
\label{bar{T}}
\bar{T}_{\Or}&\coloneqq \sum_{m=0}^{\infty}\frac{y^{2m}}{q^{2m}(1/q^2)_m}.
\end{align}
Recall the definitions of $U'_{\Or}$ and $D'_{\Or}$ in Eq. \eqref{D'_Or} and of $T_{\Sp}$ in Eq. \eqref{T_Sp}. The following factorisations for $D'_{\Or}$ and $\bar{D}'_{\Or}$ hold.
\begin{lemma}
We have
\label{fact_D'_O}
\begin{align*}
D'_{\Or}&=T_{\Sp}^{-1}(1-y)^{-1}U'_{\Or}, \\
\bar{D}'_{\Or}&=\bar{T}_{\Or}^{-1}\bar{U}'_{\Or}.
\end{align*}
\begin{proof}
We proceed as in the proof of Lemma \ref{fact_D'_u}. 
To prove the factorisation for $D'_{\Or}$, observe that, setting all variables $x_{\phi,\lambda^{\pm}}$, $x_{\phi, \lambda}$ and $x_{\bar{\phi}, \lambda}$ equal to 1 in $Z_{\Or}$ yields $\frac{1+y}{1-y}$, and
if we replace the variables $x_{z-1,\lambda^{\pm}}$ by 1 in the factor
$$\sum_{\lambda^{\pm}}x_{z-1,\lambda^{\pm}}\frac{y^{|\lambda^{\pm}|}}{c_{\Or,q}(\lambda^{\pm})}$$
of $Z_{\Or}$, the coefficient of $y^m$ in this sum becomes the sum of the proportions of unipotent elements in $\Or^{+}_{m}(q)$ and $\Or^{-}_{m}(q)$, which, by Eq. \eqref{sum_unip_qeven}, equals $\frac{1}{q^{\lfloor m/2 \rfloor}(1/q^2)_{\lfloor m/2 \rfloor}}$. Hence,
$$\sum_{\lambda^{\pm}}\frac{y^{|\lambda^{\pm}|}}{c_{\Or,q}(\lambda^{\pm})}=\sum_{m=0}^{\infty}\frac{1}{q^{\lfloor m/2 \rfloor}(1/q^2)_{\lfloor m/2 \rfloor}}y^m=T_{\Or},$$
and 
\begin{equation}
\label{e1}
D'_{\Or}=T_{\Or}^{-1}\frac{1+y}{1-y}U'_{\Or}.
\end{equation}
Now, we require a celebrated theorem of Euler \cite{Euler}, which states that
\begin{equation}
\label{euler}
\sum_{m=0}^{\infty}\frac{y^m}{q^m(1/q)_m}=\prod_{i=1}^{\infty}\left(1-\frac{y}{q^i} \right)^{-1}.
\end{equation}
We apply Eq. \eqref{euler} to $T_{\Or}$, obtaining:
\begin{align}
\label{e2}
T_{\Or}&=\sum_{h=0}^{\infty}\frac{y^{2h}}{q^h(1/q^2)_h}+\sum_{j=0}^{\infty}\frac{y^{2j+1}}{q^j(1/q^2)_j}=(1+y)\sum_{i=0}^{\infty}\frac{(q y^2)^i}{q^{2i}(1/q^2)_i}\\ &=(1+y)\prod_{i=1}^{\infty}\left(1 -\frac{qy^2}{q^{2i}}\right)^{-1}=(1+y)T_{\Sp}.
\end{align}
Combining Eq.s \eqref{e1} and \eqref{e2}, and using the factorisation of $Z_{\Or}$ in Eq. \eqref{factorisation_ci_sum_ort}, we conclude that 
$$D'_{\Or}=T^{-1}_{\Sp}(1-y)^{-1}U'_{\Or},$$
as wanted.

Similarly, to prove the statement for $\bar{D}'_{\Or}$, observe that, replacing all variables $x_{\phi,\lambda^{\pm}}$, $x_{\phi, \lambda}$ and $x_{\bar{\phi}, \lambda}$ by 1 in $Z^{[1]}_{\Or}$ we obtain 1, and if we set
all variables $x_{z-1,{\lambda^{\pm}}}$ equal to 1 in the factor
$$\sum_{\lambda^{\pm}}\tau_{z-1,\lambda^{\pm}}x_{z-1,\lambda^{\pm}}\frac{y^{|\lambda^{\pm}|}}{c_{\Or,q}(\lambda^{\pm})}$$
of $Z^{[1]}_{\Or}$, the coefficient of $y^{2m}$ in $\sum_{\lambda^{\pm}}\tau_{z-1,\lambda^{\pm}}\frac{y^{|\lambda^{\pm}|}}{c_{\Or,q}(\lambda^{\pm})}$ equals the difference between the proportions of unipotent elements in $\Or^+_{2m}(q)$ and  $\Or^-_{2m}(q)$, which, by Eq. \eqref{difference_u_ort}, is  $\frac{1}{q^{2m}(1/q^2)_m}$. Thus,
\begin{equation*}
\sum_{\lambda^{\pm}}\tau_{z-1,\lambda^{\pm}}\frac{y^{|\lambda^{\pm}|}}{c_{\Or,q}(\lambda^{\pm})}=\sum_{m=0}^{\infty}\frac{y^{2m}}{q^{2m}(1/q^2)_m}=\bar{T}_{\Or},
\end{equation*}
and 
\begin{equation*}
    \bar{D}'_{\Or}=\bar{T}^{-1}_{\Or}\bar{U}'_{\Or},
\end{equation*}
which concludes the proof.
\end{proof}
\end{lemma}
We are now ready to derive the formulas for $\delta(\AO^{\epsilon}_m(q))$ in the odd-characteristic case.
\begin{proof}[Proof of Theorem $\ref{main}\eqref{Orthogonal_odd}-\eqref{Orthogonal_even}$]
 Applying Theorem \ref{conj_identities}\eqref{conj_intro_ort_2m+1}-\eqref{conj_intro_ort_2m}, we obtain:
\begin{align*}
u'_m(q)& =\frac{1}{q^{\lfloor m/2 \rfloor}(1/q^2)_{\lfloor m/2 \rfloor}}-u_m(q) \\
& =
\begin{cases}
 \frac{1}{q^d}\sum_{i=0}^{d}\frac{(-1)^{i}}{q^{i(i-1)}(1/q^2)_{d-i}} & \text{ if }m=2d,\\
\frac{1}{q^{d+1}}\sum_{i=0}^{d}\frac{(-1)^{i}}{q^{i(i+1)}(1/q^2)_{d-i}} & \text{ if }m=2d+1.
\end{cases}
\end{align*}
Now, observe that $U'_{\Or}$ can be factorised as 
\begin{equation}
\label{fatt_ort}
    U'_{\Or}=\widetilde{D}_{\Or}\cdot T_{\Sp},
\end{equation}
where we define
$$\widetilde{D}_{\Or} \coloneqq \sum_{d=0}^{\infty}\frac{(-1)^{\lfloor\frac{d}{2}\rfloor}}{q^{\lceil \frac{d}{2}\rceil^2}}y^d.$$

Indeed,
\begin{align*}
  U'_{\Or}  =\sum_{m=0}^{\infty}u'_m(q)y^m&=\sum_{d=0}^{\infty}\sum_{i=0}^{d}\frac{(-1)^{i}}{q^{i^2}q^{d-i}(1/q^2)_{d-i}}y^{2d}+\sum_{d=0}^{\infty}\sum_{i=0}^{d}\frac{(-1)^{i}}{q^{(i+1)^2}q^{d-i}(1/q^2)_{d-i}}y^{2d+1}\\
  &= \sum_{i=0}^{\infty}\sum_{j=0}^{\infty}\frac{(-1)^i}{q^{i^2}q^j(1/q^2)_j}y^{2(i+j)}+\sum_{i=0}^{\infty}\sum_{j=0}^{\infty}\frac{(-1)^i}{q^{(i+1)^2}q^j(1/q^2)_j}y^{2(i+j)+1}\\
   &= \sum_{i=0}^{\infty}\frac{(-1)^i}{q^{i^2}}y^{2i}\sum_{j=0}^{\infty}\frac{y^{2j}}{q^j(1/q^2)_j}+\sum_{i=0}^{\infty}\frac{(-1)^i}{q^{(i+1)^2}}y^{2i+1}\sum_{j=0}^{\infty}\frac{y^{2j}}{q^j(1/q^2)_j} \\
   &=\left(\sum_{i=0}^{\infty}\frac{(-1)^i}{q^{i^2}}y^{2i}+\sum_{i=0}^{\infty}\frac{(-1)^i}{q^{(i+1)^2}}y^{2i+1}\right)\sum_{j=0}^{\infty}\frac{y^{2j}}{q^j(1/q^2)_j}=\widetilde{D}_{\Or}\cdot T_{\Sp}.
\end{align*}
Note also that 
$$\widetilde{D}_{\Or}=1+(1-y)\sum_{d=0}^{\infty}\frac{(-1)^d}{q^{(d+1)^2}}y^{2d+1}.$$
Applying Eq. \eqref{fatt_ort} and
Lemma \ref{fact_D'_O}, we find: 
\begin{align*}
D'_{\Or}&=T^{-1}_{\Sp}\cdot U'_{\Or} \cdot (1-y)^{-1}= \widetilde{D}_{\Or} \cdot (1-y)^{-1} \\
&=(1-y)^{-1}+\sum_{d=0}^{\infty}\frac{(-1)^d}{q^{(d+1)^2}}y^{2d+1}=\sum_{d=0}^{\infty}y^{2d}+\sum_{d=0}^{\infty}\left(1+\frac{(-1)^d}{q^{(d+1)^2}} \right)y^{2d+1}.
\end{align*}
Hence,
$$d'_m(q)=
\begin{cases}
1 & \text{if }m=2d, \\
1+\frac{(-1)^d}{q^{(d+1)^2}} & \text{if }m=2d+1.
\end{cases}
$$
and therefore,
\begin{equation}
\label{sum}
d_m(q)=2-d'_m(q)=
\begin{cases}
1 & \text{if }m=2d, \\
1+\frac{(-1)^{d-1}}{q^{(d+1)^2}} & \text{if }m=2d+1.
\end{cases}
\end{equation}
For the difference case, similarly, applying Proposition \ref{prop_bar{u}} we obtain:
\begin{align*}
\bar{u}'_{2m}(q)& =\frac{1}{q^{2m}(1/q^2)_m}-\bar{u}_{2m}(q) \\
& =\frac{1}{q^{2m}}\sum_{i=0}^{m}\frac{(-1)^{i}}{q^{i(i-1)}(1/q^2)_{m-i}}.
\end{align*}

Using Lemma \ref{fact_D'_O}, we have $U'_{\Or}=\bar{D}'_{\Or}\cdot \bar{T}_{\Or}$. Observe that
\begin{align*}
\bar{U}'_{\Or}=\sum_{m=0}^{\infty}\sum_{i=0}^{m}\frac{(-1)^{i}}{q^{i(i+1)}q^{2(m-i)}(1/q^2)_{m-i}}y^{2m}&=\sum_{i=0}^{\infty}\frac{(-1)^i}{q^{i(i+1)}}y^{2i}\sum_{j=0}^{\infty}\frac{1}{q^{2j}(1/q^2)_j}y^{2j}\\
& =\sum_{i=0}^{\infty}\frac{(-1)^i}{q^{i(i+1)}}y^{2i}\cdot \bar{T}_{\Or}.
\end{align*}
Hence, 
$$\bar{D}'_{\Or}=\sum_{m=0}^{\infty}\frac{(-1)^m}{q^{m(m+1)}}y^{2m},$$
and for every $m \ge 1$,
$$\bar{d}'_{2m}(q)=\frac{(-1)^m}{q^{m(m+1)}}.$$
Therefore
\begin{equation}
\label{difference}
\bar{d}_{2m}(q)=-\bar{d}'_{2m}(q)=\frac{(-1)^{m-1}}{q^{m(m+1)}}.
\end{equation}
In view of Equations \eqref{sum} and \eqref{difference}, we can finally conclude that, if $m=2d+1$ is odd, then
$$\delta(\AO_{2d+1}(q))=\frac{d_{2d+1}(q)}{2}=\frac{1}{2}+\frac{(-1)^{d-1}}{2q^{(d+1)^2}},$$
and if $m=2d$ is even, then
\begin{equation*}
\delta(\AO^{\pm}_{2d}(q))=\frac{d_{2d}(q)\pm\bar{d}_{2d}(q)}{2}=\frac{1}{2}\pm \frac{(-1)^{d-1}}{2q^{d(d+1)}}.
\qedhere
\end{equation*}
\end{proof}

\section{Affine orthogonal groups in even characteristic}
Throughout this section, $q=2^f$ is a $2$-power.
As discussed in \cite[Sec. 4]{FulmanSaxlTiep12}, Wall \cite{Wall} proved that the $\GL_{2m}(q)$-rational canonical form of an element $a \in \Or_{2m}^{\pm}(q)$ is described by the following combinatorial data.
To each monic, non-constant irreducible polynomial $\phi$ over $\mathbb{F}_q$, the element $a$ associates a partition $\lambda_{\phi}=\lambda_{\phi}(a)$, determined by its rational canonical form. Recall the definition of $\bar{\phi}$ given in Eq. \eqref{bar_phi}. The collection $(\lambda_{\phi})_{\phi}$ represents a conjugacy class in either $\Or_{2m}^{+}(q)$ or $\Or_{2m}^{-}(q)$ if 
\begin{enumerate}[label=(\roman*)]
    \item $|\lambda_z|=0,$
    \item $\sum_{\phi} |\lambda_{\phi}|\deg (\phi)=2m,$
    \item $\lambda_{\phi}=\lambda_{\bar{\phi}},$
    \item \text{the odd parts of $\lambda_{z-1}$ occur with even multiplicity}.
\end{enumerate}

The cycles indices of the  sum and the difference of the orthogonal groups in even characteristic are

$$1+\sum_{m=1}^{\infty}\frac{y^{2m}}{|\Or^{+}_{2m}(q)|}\sum_{a \in \Or^{+}_{2m}(q)}\prod_{\phi}x_{\phi,\lambda_{\phi}(a)}\pm\sum_{m=1}^{\infty}\frac{y^{2m}}{|\Or^{+}_{2m}(q)|}\sum_{a \in \Or^{-}_{2m}(q)}\prod_{\phi}x_{\phi,\lambda_{\phi}(a)}.
$$
We denote the cycle index of the sum by $Z_{\Or(2)}$ (that is, the quantity obtained by choosing the sign $+$ in the expression above) and the cycle index  of the difference by $Z^{[1]}_{\Or(2)}$,  in accordance with the notation used for the odd characteristic case.
In \cite{FulmanSaxlTiep12}, the following factorizations for $Z_{\Or(2)}$ and $Z^{[1]}_{\Or(2)}$ were obtained.

\begin{thm}[\cite{FulmanSaxlTiep12}, Theorem 4.1]
\begin{align*}
Z_{\Or(2)}=&\sum_{\substack{|\lambda|=2m\\ i \text{ odd }\Rightarrow m_i \text{ even}}}x_{\phi,\lambda}\frac{y^{|\lambda|}}{q^{\frac{1}{2}\sum_i(\lambda'_i)^2+\frac{1}{2}o(\lambda)-\lambda'_1}\prod_i(1/q^2)_{\lfloor \frac{m_i(\lambda)}{2} \rfloor}}
\\ &\cdot \prod_{\substack{\phi=\bar{\phi}\\\phi \neq z-1}}\sum_{\lambda}x_{\phi,\lambda}\frac{(-y^{\deg \phi)|\lambda|}}{c_{\GL,-q^{\deg \phi/2}}(\lambda)} \nonumber
\cdot
\prod_{\substack{\{\phi,\bar{\phi}\}\\\phi \neq \bar{\phi}}}\sum_{\lambda}x_{\phi,\lambda}x_{\bar{\phi},\lambda}
\frac{y^{2|\lambda|\deg \phi}}{c_{\GL,q^{\deg \phi}}(\lambda)}; \\
Z^{[1]}_{\Or(2)}=&\sum_{\substack{|\lambda|=2m\\ \text{all $m_i$ even}}}x_{\phi,\lambda}\frac{y^{|\lambda|}}{q^{\frac{1}{2}\sum_i(\lambda'_i)^2}\prod_i(1/q^2)_{ \frac{m_i(\lambda)}{2}}}
\\ &\cdot \prod_{\substack{\phi=\phi^*\\\phi \neq z-1}}\sum_{\lambda}x_{\phi,\lambda}\frac{y^{\deg \phi|\lambda|}}{c_{\GL,-q^{\deg \phi/2}}(\lambda)} \nonumber
\cdot
\prod_{\substack{\{\phi,\bar{\phi}\}\\\phi \neq \bar{\phi}}}\sum_{\lambda}x_{\phi,\lambda}x_{\bar{\phi},\lambda}
\frac{y^{2|\lambda|\deg \phi}}{c_{\GL,q^{\deg \phi}}(\lambda)}.
\end{align*}
    
\end{thm}

\subsection{Derangements of 2-power order: proof of Theroem \ref{p_main}\eqref{p_Orthogonal_2meven}}
\begin{prop}
\label{u'_2m(q)orth}
The following hold:
\begin{enumerate}[label=\roman*.]
    \item 
$\displaystyle u'_{2m}(q)=\sum_{\substack{|\lambda|=2m\\i \text{ odd}\Rightarrow m_i \text{ even}}}\frac{1}{q^{\frac{1}{2}\sum_i (\lambda'_i)^2+\frac{1}{2}o(\lambda)}\prod_i(1/q^2)_{\lfloor\frac{m_i}{2}\rfloor}}=\frac{1}{q^m(1/q^2)_m},$ \\
that is, $u'_m(q)$ is equal to the proportion of unipotent elements of $\Sp_{2m}(q)$;
\item 
$\displaystyle
\bar{u}_{2m}(q)=\sum_{\substack{|\lambda|=2m\\ \text{all }m_i \text{ even}}}\frac{1-q^{-\lambda'_1}}{q^{\frac{1}{2}\sum_i (\lambda'_i)^2}\prod_i(1/q^2)_{\frac{m_i}{2}}}=\frac{1}{q^{2m}}\sum_{i=1}^{m}\frac{(-1)^{i-1}}{q^{i(i-1)}(1/q^2)_{m-i}}.$
\end{enumerate}
\begin{proof}
It follows from the definition of the cycle index $Z_{\Or(2)}$ that $u'_{2m}(q)$ is the coefficient of $y^{2m}$ in 
$$\sum_{\substack{|\lambda|=2m\\ i \text{ odd }\Rightarrow m_i \text{ even}}}x_{z-1,\lambda}\frac{y^{|\lambda|}}{q^{\frac{1}{2}\sum_i(\lambda'_i)^2+\frac{1}{2}o(\lambda)-\lambda'_1}\prod_i(1/q^2)_{\lfloor\frac{m_i}{2}\rfloor}}$$
when we substitute all variables $x_{z-1,\lambda}$ with $q^{-\lambda'_1}$. This gives the first equality for $u'_{2m}(q)$. To obtain the second equality, note that
$$\sum_{\substack{|\lambda|=2m\\i \text{ odd}\Rightarrow m_i \text{ even}}}\frac{1}{q^{\frac{1}{2}\sum_i (\lambda'_i)^2+\frac{1}{2}o(\lambda)}\prod_i(1/q^2)_{\lfloor\frac{m_i}{2}\rfloor}}$$
is the coefficient of $y^{2m}$ in the factor
$$\sum_{\substack{|\lambda|=2m\\ i \text{ odd }\Rightarrow m_i \text{ even}}}x_{z-1,\lambda}\frac{y^{|\lambda|}}{q^{\frac{1}{2}\sum_i(\lambda'_i)^2+\frac{1}{2}o(\lambda)
}\prod_i(1/q^2)_{\lfloor\frac{m_i}{2}\rfloor}}$$
of $Z_{\Sp}$, when we replace all variables $x_{z-1,\lambda}$ by 1, and from the definition of this cycle index, this is exactly the proportion of unipotent elements in $\Sp_{2m}(q)$, which equals $\frac{1}{q^m(1/q^2)_m}$ by Lemma \ref{steinberg}\eqref{p_sp}.

The second statement follows from Proposition \ref{prop_bar{u}}, as  $\bar{u}_{2m}(q)$ has the same expression regardless of the characteristic of $\mathbb{F}_q$.
\end{proof}
\end{prop}

\begin{proof}[Proof of Theorem \ref{p_main}\eqref{p_Orthogonal_2meven}]
Recalling the definition of $u'_{2m}(q)$, from Eq. \eqref{sum_unip_qeven} and Proposition \ref{u'_2m(q)orth} we have
\begin{align*}
u_{2m}(q)&=\Delta_u(\Or^{+}_{2m}(q))+\Delta_u(\Or^{-}_{2m}(q))-u'_{2m}(q)\\&=\frac{1}{q^m(1/q^2)_m}+\frac{1}{q^{m-1}(1/q^2)_{m-1}}-\frac{1}{q^m(1/q^2)_m}=\frac{1}{q^{m-1}(1/q^2)_{m-1}}.
\end{align*}
Hence,
\begin{align*}
\delta_2(\AO^{\pm}(q))&=\frac{u_{2m}(q)\pm\bar{u}_{2m}(q)}{2}\\&=\frac{1}{2q^{m-1}(1/q^2)_{m-1}}\pm \frac{1}{2q^{2m}}\sum_{i=1}^{m}\frac{(-1)^{i-1}}{q^{i(i-1)}(1/q^2)_{m-i}}. \qedhere
\end{align*}
\end{proof}

\subsection{Derangements: proof of Theorem \ref{main}\eqref{Orthogonal_even} }
Recall the definition of $\bar{T}_{\Or}$ in Eq. \eqref{bar{T}}.
\begin{lemma}
We have
\label{fact_D'_Oeven}
\begin{align*}
D'_{\Or}=\frac{1}{1-y^2}, \\
\bar{D}'_{\Or}=\bar{T}_{\Or}^{-1}\bar{U}'_{\Or}.
\end{align*}
\begin{proof}
We prove the equality for $D'_{\Or}$. The proof of the equality for $\bar{D}'_{\Or}$ is the same as in Lemma \ref{fact_D'_O}, as all quantities are the same as in the odd-characteristic case.

Let \begin{align*}
T_{\Or(2)}\coloneqq& \sum_{m=0}^{\infty}\left(\Delta_u(\Or^{+}_{2m}(q))+\Delta_u(\Or^{-}_{2m}(q))\right)y^{2m}\\=&\sum_{m=0}^{\infty}\left(\frac{1}{q^m(1/q^2)_m}+\frac{1}{q^{m-1}(1/q^2)_{m-1}}\right)y^{2m} 
\end{align*}
be the generating function for the sum of the proportions of unipotent elements in $\Or^{+}_{2m}(q)$ and $\Or^{-}_{2m}(q)$. Recalling the definition of $T_{\Sp}$ in Eq. \eqref{T_Sp},
\begin{equation}
\label{T_O=(1+y^2)T_Sp}
T_{\Or(2)}=(1+y^2)T_{\Sp}.
\end{equation}
Observe that, setting all variables $x_{\phi, \lambda}$ and $x_{\bar{\phi}, \lambda}$ equal to 1 in $Z_{\Or(2)}$ yields $\frac{1+y^2}{1-y^2}$, and
if we replace the variables $x_{z-1,\lambda}$ by 1 in the factor
$$\sum_{\substack{|\lambda|=2m\\ i \text{ odd }\Rightarrow m_i \text{ even}}}x_{\phi,\lambda}\frac{y^{|\lambda|}}{q^{\frac{1}{2}\sum_i(\lambda'_i)^2+\frac{1}{2}o(\lambda)-\lambda'_1}\prod_i(1/q^2)_{\lfloor \frac{m_i}{2} \rfloor}}$$
of $Z_{\Or(2)}$, the coefficient of $y^{2m}$ in this sum equals $\Delta_u(\Or^{+}_{2m}(q))+\Delta_u(\Or^{-}_{2m}(q))$. Hence,
$$D'_{O(2)}=T^{-1}_{O(2)}\frac{1+y^2}{1-y^2}U'_{O(2)}.$$
By Proposition \ref{u'_2m(q)orth}, $U'_{\Or(2)}=T_{\Sp}$ and applying Eq. \eqref{T_O=(1+y^2)T_Sp}, we conclude that
$$D'_{O(2)}=\frac{1}{1-y^2}.$$
\end{proof}
\end{lemma}
\begin{proof}[Proof of Theorem \ref{main}\eqref{Orthogonal_even}]
By Lemma \ref{fact_D'_Oeven}, $d'_{2m}(q)$ and $\bar{d}'_{2m}(q)$ have the same expression regardless of the parity of the characteristic of $\mathbb{F}_q$. Hence, as in the odd-characteristic case, we conclude that
$$\delta(\AO^{\pm}_{2m}(q))=\frac{1}{2} \pm \frac{(-1)^{m-1}}{2q^{m(m+1)}}.$$
\end{proof}

\bibliographystyle{plain}
\bibliography{bibliography}

\end{document}